\documentclass[12pt,oneside]{amsart}
\usepackage{etex}

\usepackage[breakwords]{truncate}

\usepackage{hyperref}
\usepackage{xcolor}
\hypersetup{
    colorlinks,
    linkcolor={blue!80!black},
    citecolor={blue!50!black},
    urlcolor={blue!80!black}
}

\usepackage{amsfonts,amssymb,amscd,amsmath,latexsym,amsbsy, setspace}
\usepackage{amsfonts,amssymb,graphicx,amsthm,mathtools}
\usepackage{fullpage}
\usepackage[all,cmtip,curve, knot,frame]{xy} 
\usepackage{amssymb, amsmath}
\usepackage{amsfonts}
\usepackage{latexsym}
\usepackage{epstopdf}
\usepackage{tikz}
\usepgflibrary{shapes.geometric}
\usetikzlibrary{matrix}
\usepackage{pgfplots}
\usepackage{braids}
\usepackage{geometry}

\usepackage{soul} 

\newtheorem{theorem}{Theorem}[section]
\newtheorem{lemma}[theorem]{Lemma}

\newtheorem{proposition}[theorem]{Proposition}
\newtheorem{corollary}[theorem]{Corollary}
\theoremstyle{definition}
\newtheorem{definition-proposition}[theorem]{Definition-Proposition}
\newtheorem{definition}[theorem]{Definition}
\newtheorem{example}[theorem]{Example}
\newtheorem{remark}[theorem]{Remark}

\newcommand{\RV}{\mf R}
\newcommand{\oa}{\mathfrak{F}_\mathcal{A}}
\newcommand{\oak}[1]{\mathfrak{F}_{\mathcal{A}^{(#1)}}}
\newcommand{\act}{\operatorname{act}}
\newcommand{\coker}{\operatorname{coker}}

\newcommand{\LFP}{{\operatorname{\mathbf{Pr_c}}}}
\newcommand{\Pres}{\operatorname{\mathbf{Pr}}}
\newcommand{\ind}{\operatorname{ind}}
\newcommand{\comp}{\operatorname{comp}}
\newcommand{\h}{\hbar}
\newcommand{\mf}{\mathfrak}

\newcommand{\Vect}{\operatorname{Vect}}

\newcommand{\End}{\operatorname{End}}
\newcommand{\Hom}{\operatorname{Hom}}
\newcommand{\iHom}{\underline{\operatorname{Hom}}}

\newcommand{\Rep}{\operatorname{Rep}}
\newcommand{\Repq}{\operatorname{Rep}_q\!}

\newcommand{\strike}[1]{}

\newcommand{\K}{\mathbf{k}}

\newcommand{\cA}{\mathcal{A}}
\newcommand{\cB}{\mathcal{B}}
\newcommand{\cC}{\mathcal{C}}
\newcommand{\cD}{\mathcal{D}}

\newcommand{\cO}{\mathcal{O}}

\newcommand{\cM}{\mathcal{M}}

\newcommand{\cN}{\mathcal{N}}

\newcommand{\QCoh}{\operatorname{QCoh}}

\newcommand{\Rex}{\operatorname{\mathbf{Rex}}}

\newcommand{\Lex}{\operatorname{\mathbf{Lex}}}
\newcommand{\Ab}{\operatorname{\mathbf{Ab}}}
\newcommand{\cE}{\mathcal{E}}
\newcommand{\dS}{\Big/\hspace{-5pt}\Big/}

\newcommand{\g}{\mathfrak{g}}

\newcommand{\ot}{\otimes}
\newcommand{\bt}{\boxtimes}
\newcommand{\id}{\operatorname{id}}

\newcommand{\RR}{\mathbb{R}}
\newcommand{\CC}{\mathbb{C}}

\newcommand{\ZZ}{\mathbb{Z}}

\newcommand{\oaH}{\widetilde H}
\newcommand{\oaHk}[1]{\widetilde H_{#1}}

\newcommand{\oo}{\infty}

\newcommand{\cat}{\operatorname{Cat}}

\newcommand{\uch}{\underline{\operatorname{Ch}}}
\newcommand{\ch}{\operatorname{Ch}}
\newcommand{\modu}{\operatorname{-mod}}

\newcommand{\modul}{\operatorname{mod}}

\newcommand{\un}{\mathbf{1}}

\newcommand{\Gr}{\operatorname{\mathbf{Gr}}}

\usepackage[disable]{todonotes}

\usepackage{environ}
\NewEnviron{replaces}{}

\def\HH{\hbox{${\mathcal H}$\kern-5.2pt${\mathcal H}$}}

\numberwithin{equation}{section}

\begin{document}
\title[Integrating quantum groups]{Integrating quantum groups over surfaces}
\author{David Ben-Zvi}\address{Department of Mathematics\\University
  of Texas\\Austin, TX 78712-0257} \email{benzvi@math.utexas.edu}
\author{Adrien Brochier} \address{MPIM, Bonn}\email{abrochier@mpim-bonn.mpg.de}
\author{David Jordan} \address{School of Mathematics\\University of Edinburgh\\ Edinburgh, UK}\email{D.Jordan@ed.ac.uk}

\maketitle
\begin{abstract}
We apply the mechanism of factorization homology to construct and compute category-valued two-dimensional topological field theories associated to braided tensor categories, generalizing the $(0,1,2)$-dimensional part of Crane-Yetter-Kauffman 4D TFTs associated to modular categories.   Starting from modules for the Drinfeld-Jimbo quantum group $U_q(\g)$ we obtain in this way an aspect of topologically twisted 4-dimensional ${\mathcal N}=4$ super Yang-Mills theory, the setting introduced by Kapustin-Witten for the geometric Langlands program. 

For punctured surfaces, in particular, we produce explicit categories which quantize character varieties (moduli of $G$-local systems) on the surface; these give uniform constructions of a variety of well-known algebras in quantum group theory.  From the annulus, we recover the reflection equation algebra associated to $U_q(\g)$, and from the punctured torus we recover the algebra of quantum differential operators associated to $U_q(\g)$.  From an arbitrary surface we recover Alekseev's moduli algebras. Our construction gives an intrinsically topological explanation for well-known mapping class group symmetries and braid group actions associated to these algebras, in particular the elliptic modular symmetry (difference Fourier transform) of quantum $\cD$-modules.
\end{abstract}

\tableofcontents

\section{Introduction}\label{sec:introduction}

In this paper we develop factorization homology valued in braided tensor categories and apply it to explicitly construct and compute category-valued invariants of topological surfaces. The most important example for us is the braided tensor category $\Repq G$ of integrable representations for the quantum group associated to the reductive group $G$: from these we obtain categories which we call \emph{quantum character varieties}.  These quantize moduli spaces of local systems on surfaces and provide a unifying perspective on various constructions in quantum group theory. Quantum character varieties form the $2-$dimensional part of a topological field theory which is a model for the Kapustin-Witten theory~\cite{Kapustin2007} (GL-twisted 4-dimensional ${\mathcal N}=4$ super Yang-Mills theory), and provide the spectral side of the quantum Betti geometric Langlands conjecture~\cite{Ben-Zvi2016a}. These connections (which are discussed further in Sections~\ref{TFT section} and~\ref{Betti section}) suggest many rich structures for quantum character varieties, some of which we discuss in this paper, and many which we plan to explore in future papers.

\subsection{Factorization homology of surfaces}\label{sec:intro:facto}
Factorization homology was originally introduced by Beilinson and Drinfeld~\cite{Beilinson2004} in the setting of chiral conformal field theory, as an abstraction (and geometric interpretation) of the functor of conformal blocks of a vertex algebra. Factorization homology in the topological, rather than conformal, setting is developed in~\cite{Lurie} and further in~\cite{Ayala2015, Ayala2017}.  In this paper we use the terminology and formalism of~\cite{Ayala2015, Ayala2017}. See~\cite{Ginot2015} for a survey and~\cite{Costello} for more general applications to quantum field theory).

The algebraic input to factorization homology of surfaces (in the terminology of~\cite{Ayala2015} and subsequent papers) is a ``2-disk algebra" in an appropriate symmetric monoidal higher category $\cC^\bt$. Informally speaking, a 2-disk algebra is an object $\cA\in \cC$ equipped with operations $\cA^{\bt k}\to \cA$ parametrized in a locally constant fashion by embeddings of disjoint unions of $k$ disks into a large disk, and satisfying a composition law governed by composition of disk embeddings. There are several variants of 2-disk algebras, named according to the kind of tangential structure carried by the disks and embeddings: framed 2-disk algebras are better known as $E_2$ algebras (algebras over the little 2-disk operad), while oriented 2-disk algebras are (confusingly) known as framed $E_2$-algebras (algebras over the framed little 2-disk operad). We adopt the terminology of~\cite{Ayala2015} as it reflects the type of surfaces over which the corresponding algebras may be integrated.

Our approach to constructing quantum character varieties is to apply the mechanism of factorization homology in the nonlinear setting of categories rather than the linear setting of vector spaces or chain complexes, i.e., our target is a certain (2,1)-category $\cC^{\bt}$ of $\K$-linear categories, linear functors and natural isomorphisms.\footnote{Let us delay until Section \ref{sec:category} a discussion of the precise 2-categorical framework.}  An elementary but important observation is that in this case, an $E_2$-structure is determined already by the binary product $\cA^{\bt 2}\to \cA$ labeled by a fixed embedding, an associator natural isomorphism, and a collection of braid group actions, given by monodromy over the configuration space. Taking into account compositions and coherences, one finds that an $E_2$-algebra in categories is simply a braided tensor category.  Likewise, a 2-disk algebra in categories is a {\em balanced} (braided) tensor category (e.g. a pivotal braided category; see Remark \ref{rmk:balanced}).

The factorization homology of a framed surface $S$ with coefficients in an $E_2$-algebra $\cA$ is another object of $\cC^\bt$, denoted $\int_S \cA$.  The assignment,
\[
S \longmapsto \int_S \cA  \in \cC^\bt,
\]
induces a symmetric monoidal functor from a certain category of framed manifolds and framed embeddings to $\cC^{\bt}$.   When $\cA$ is an oriented 2-disk algebra, this functor becomes independent, up to equivalence, of the framing and then descends to a functor from the analogous category of \emph{oriented} manifolds and oriented embeddings. 

The invariant thus produced may be characterized as follows: to an open disk $D^2$ is assigned, by definition, the underlying object of $\cA$; to a general surface $S$, is assigned the ``integration" (i.e. co-limit) over all possible embeddings of a finite disjoint union of disks $i:(D^2)^{\sqcup k}\hookrightarrow S$.  By construction the factorization homology of any surface $S$ with coefficients in $\cA$ carries a universal morphism, $$\Delta_i:\cA^{\bt k}\to \int_S\cA,$$ for every disk embedding, which moreover factors through the $E_2$-multiplication (in our case, braided tensor structure), whenever it factors through a larger disk embedding.  Finally, the unit of $\cA$ endows the factorization homology of any surface with a canonical pointing: there is a distinguished object $$\cO_{\cA,S}\in \int_S\cA$$ which can be realized as $$\cO_{\cA,S}\simeq \Delta_i(1_\cA),$$ for any disk embedding $i:D^2\hookrightarrow S$. 

Factorization homology is a homology theory in the sense that it satisfies an \emph{excision} property which is a primary tool for computations.  Given an oriented 1-manifold $M$ (in particular $M=S^1$ or $M=I$), the factorization homology $\int_{M\times \RR} \cA$ carries a canonical $E_1$-structure (i.e., an associative product; in our categorical setting, a tensor product) from the inclusion of disjoint unions of intervals inside a larger interval (i.e., we stack cylinders inside a larger cylinder). Moreover the invariant of a manifold with a collared boundary $M$ is naturally a module for the $E_1$-structure on the invariant of its collared boundary (see Section~\ref{sec:examples} for a number of figures illustrating excision in examples). 
This structure gives rise to the \emph{excision} property of factorization homology: 
let $S=X^- \underset{M\times \RR}{\sqcup} X^+$ be a collar gluing of $X^{\pm}$ along a 1-manifold $P$ with a trivialization $N\cong M\times \RR$ of a tubular neighborhood $N$ of $M$. Then $\int_{X^{\pm}} \cA$ are left and right $\int_{M \times \RR}\cA$-modules respectively, and there is an equivalence \cite{Ayala2017,Ayala2015} of categories
	\[
		\int_S \cA \simeq \int_{X^-} \cA \underset{\int_{M\times\RR} \cA}{\boxtimes}  \int_{X^+} \cA.
	\]

\subsection{Summary of results} Fix a balanced abelian rigid braided tensor category $\cA$.  Our main results are as follows:

\begin{enumerate}
	\item[$\bullet$] We construct the category $\int_S \cA$ for any oriented surface $S$, equipped with a distinguished object $\cO_{\cA,S}$ and carrying an action of the group of orientation-preserving diffeomorphisms of $S$. 

	\item[$\bullet$] To a surface with distinguished $S^1$ boundary component together with a point chosen on it we attach a canonical algebra object $$A_S:=\underline{\End}_\cA(\cO_{\cA,S})\in \cA,$$ the internal endomorphism algebra of $\cO_{\cA,S}$. In Theorem \ref{thm:main} we produce an equivalence, $$\int_S\cA\simeq A_S\modu_{\cA},$$ as a module category for $\cA$.

\item[$\bullet$] In Section~\ref{sec:surfaces}, we develop a combinatorial framework which allows us to produce explicit presentations of the algebras $A_S$ from a combinatorial presentation (``handle and comb decomposition") of a punctured surface $S$.  This involves giving an explicit computation of relative tensor products dictated by the excision axiom.

\end{enumerate}

If we drop the assumption that $\cA$ is balanced, we obtain an invariant of framed surfaces.   We also prove analogous results to each of the above.  While our results apply to general abelian rigid braided tensor categories, we will mostly be concerned in this paper with the example $\cA=\Repq G$ of representations of the quantum group $U_q\mathfrak{g}$, associated to a reductive group $G$ and an arbitrary $q\in\CC^\times$ (see Section \ref{qg-defs} for a precise definition).  Note that the $A_S$-modules which appear in this case are required to be locally finite as modules for $U_q(\mathfrak{g})$, however they are not in any sense locally finite for the $A_S$-action.
In the special case of the symmetric tensor category $\cA=\Rep G$, the factorization homology $\int_S \Rep G$ makes sense on {\em any} topological space $S$.

In the derived setting, it was proved in~\cite{Ben-Zvi2010} that the result is the dg category of quasi-coherent sheaves on the  {\em character stack} $\uch_G(S)$, the moduli stack of $G$-local systems on $S$ (parametrizing homomorphisms $\rho:\pi_1(S)\to G$, modulo conjugation in $G$):
$$ \int_S \Rep_{dg} G \simeq QC_{dg}(\uch_G(S)).$$  

\begin{remark} More precisely, this statement holds for the natural derived enhancement of the character stack. The difference between the derived and underived character stacks does not affect the abelian categories of quasi-coherent sheaves, and so will be invisible for the constructions in this paper. Thus the reader is invited to interpret $\uch_G(S)$ in derived or underived fashion, without affecting the results discussed.
\end{remark}

The analogous result for abelian categories of quasi-coherent sheaves on character stacks of surfaces follows from the main results of this paper (see Remark~\ref{derived} for the relation between the derived and abelian theories).
Thus we will consider the category $\int_S\Repq G$ to be the quantum analog of sheaves on the character stack, and refer to it as the ($G$-){\em quantum character variety}. In this particular we have the following:  
\begin{enumerate}
	\item[$\bullet$] For any surface $S$ the subalgebra of invariants of $A_S$ is an explicit  quantization of the Poisson algebra of functions on the $G$-character variety $\ch_G(S),$  equivariant for the action of the mapping class group of $S$, (Section \ref{sec:CharVar}).

	\item[$\bullet$] For the annulus $Ann$ the algebra $A_{Ann}$ recovers the reflection equation algebra $\cO_q(G)$.  For the punctured torus $A_S$ recovers the algebra $\cD_q(G)$ of quantum differential operators on $G$.  More generally, for a punctured surface of arbitrary genus we  recover the so-called moduli algebras of Alekseev.  See Section \ref{sec:examples} for a comprehensive review of these examples.
\end{enumerate}

\subsubsection{Factorization homology of linear categories}
In order to set up factorization homology of linear categories, we require a suitable higher categorical framework --- specifically we require a higher category of categories which carries a symmetric monoidal structures and is closed under (sufficiently many) colimits. The collection of abelian categories fails to satisfy these properties, and must be enlarged, in one of several related ways. In setting up the general foundations we establish a formalism that is potentially of independent interest to specialists in the theory of tensor categories. 

\begin{enumerate}
	\item[$\bullet$] We define factorization homology of surfaces with coefficients in rigid braided tensor categories in the $2$-category $\LFP$ of compactly generated presentable categories (or equivalently $\Rex$ of finitely cocomplete categories).

\item[$\bullet$] We develop several techniques, related to Beck monadicity, for describing module categories explicitly as categories of modules for algebra objects, mostly extending well-known results from the setting of finite (and typically semi-simple) tensor categories (cf.~\cite{Etingof2010,Ostrik2003,Douglas2013}) to the infinite and non-semisimple setting.  In particular we prove monadicity results describing various constructions with abelian tensor categories (such as base changes, relative tensor products, and traces) as categories of modules over internal endomorphism algebras. 

\end{enumerate}

\begin{remark}[Derived version]\label{derived}
Prompted by discussions following the posting of our preprint, J.~Lurie~\cite{Lurie2016} has proved that (contrary to widely held expectations) the collection of Grothendieck abelian categories is closed under tensor products. The result is part of a general theory of derived analogues, called Grothendieck prestable $\infty$-categories, which are roughly (the positive halves of) stable $\infty$-categories (or in the $k$-linear setting where we work, $k$-linear differential graded categories) with t-structures whose heart is a Grothendieck abelian category. It follows from Lurie's results that one can define factorization homology of $E_2$-categories either in the setting of dg (or stable $\infty$) categories, or in the refined setting of Grothendieck prestable categories, i.e., keeping track of t-structures.  For rigid braided tensor categories one can check that our constructions in the discrete setting are compatible with their derived analogues -- i.e., our quantum character varieties are the hearts of the natural derived quantum character varieties, obtained by integrating the dg categories of representations of quantum groups and keeping track of $t$-structures. (We intend to return to this derived comparison in a later paper.) In particular this implies {\it a priori} that the categories constructed in this paper are in fact abelian -- a fact which is evident from our explicit description.
\end{remark}

\begin{remark}[Skein Categories]
	It is interesting to compare factorization homology with the theory of skein algebras and skein categories (see e.g.~\cite{Roche2002,Frohman2000,Walker}), which provide a convenient graphical calculus for quantizing $SL_2$- (and more generally $SL_n$-) character varieties and for constructing associated 3-manifold invariants.  
Roughly speaking, to each surface $S$ and each braided tensor category $\cA$ with a choice of presentation (i.e. a collection of objects generating under tensor product, morphisms generating under composition, and a specification of ``local" relations on morphisms), there is an associated skein category in which an object is a configuration of disks, each colored by a generating object of $\cA$, and in
which a morphism is the quotient of the vector space of colored tangles
by local relations in $\cA$.  The result is not typically an abelian category.  The skein algebra is the endomorphism algebra of the empty configuration, so it consists of closed tangles drawn in $S\times I$.

There is an evident functor to the corresponding factorization homology category, and so one may regard the latter as a co-completion of the former, though in what precise sense this should be meant is beyond the scope of the present paper.  

Nevertheless, this perspective gives a very useful dictionary:  While skein categories are elementary to define, they can be difficult to compute with algebraically, and do not have good functoriality properties.  Factorization homology on the other hand enjoys manifest locality and functoriality properties, independence of presentation, good behavior for closed as well as open surfaces and a natural (and nontrivial) derived extension, which make it readily applicable to problems in geometric representation theory. It would be very interesting to identify our factorization homology construction and the similarly general abstract skein theory developed by Walker.
\end{remark}

\subsection{Detailed overview}
We now describe our main results in more detail.

\subsubsection{Punctured surfaces}

Our main result, Theorem \ref{thm:main}, is a concrete computation of the quantum character variety of an arbitrary punctured surface $S$.  
Recall that the inclusion $i:D^2\hookrightarrow S$ of disk in a surface defines a  functor $$\Delta_i:\cA\to \int_S \cA, \hskip.3in \mathbf{1}_\cA\mapsto \cO_{\cA,S}$$ from $\cA$ to the quantum character variety, sending the unit $\mathbf{1}_\cA$ to the ``quantum structure sheaf", the distinguished object $\cO_{\cA,S}\in\int_S\cA$. 
In the commutative case $\cA=\Rep G$ this functor is given by pulling back sheaves under $$\uch_G(S)\longrightarrow\uch_G(D^2)\simeq pt/G,$$ and the distinguished object is the structure sheaf $\cO_{\uch_G(S)}$ of the classical character stack. Unlike the commutative case, however, the quantum character variety does not carry an $\cA$-module structure in general. This is the two-dimensional analog of the one-dimensional assertion that the cocenter $A/[A,A]$ (or Hochschild homology in the derived setting) $\int_{S^1}A$ of an associative algebra $A$ carries a trace map from $A$ but has no natural $A$-module structure.

Our description of quantum character varieties is based on the observation that if we chose a boundary interval $I\subset \partial S$ we obtain a natural $\cA$-module structure on $\int_S\cA$. 
\begin{remark} In fact this choice equips $\int_S\cA$ with the structure moreover of a \emph{braided module category}; this structure will be developed and exploited in a follow-up paper\cite{Ben-Zvi2016}.\end{remark}

We can then describe the entire category $\int_S\cA$ as the category of modules in $\cA$ for the internal endomorphism algebra of the distinguished object, relative to this $\cA$-module structure, as follows.

Let $A_S:=\underline{\End}_\cA(\cO_{\cA,S})$ denote the internal endomorphism algebra of $\cO_{\cA,S}$.  Here, the internal endomorphism algebra, $\underline{\End}_\cA(m)$ is an algebra in $\cA$, determined by an $\cA$-module category $\cM$ and an object $m$ in $\cM$.  See Section~\ref{sec:tensorcat} for more details.

\begin{theorem}[Theorem~\ref{thm:main}]  Let $\cA$ be a rigid abelian braided tensor category in $\LFP$.  We have an equivalence of categories,
$$\int_S\cA \simeq A_S\modu_\cA$$ 
respecting natural actions of the mapping class group of $S$ relative to the boundary.
\end{theorem}

\begin{remark} For $\cA=\Rep G$, this captures the statement that $\uch_G(S)$ is affine over $pt/G$.\end{remark}

		The algebra $A_S$ can be described in completely explicit terms, once one chooses a combinatorial presentation of the surface $S$, which we call a ``gluing pattern'' $P$ (see Section~\ref{sec:moduliAlgebra} and Figure~\ref{gluing-pattern} therein for a quick sense of what we mean). Suppose the surface $S$ has genus $g$ and $r\geq 1$ punctures, so that its fundamental group is free of rank $2g+r-1$. The gluing pattern highlights a set of free generators, and for each rigid braided tensor category $\cA$ determines an algebra $a_P \in \cA$. Loosely speaking this algebra is built from $2g+r-1$ copies of the reflection equation algebra $\oa=A_{Ann}$ (see Definition~\ref{def:REA}; it only depends on $\cA$), but with relations among these ``generators'' determined by $P$. Explicitly, $ a_P\cong (\oa)^{\ot 2g+r-1}$ where each tensor factor is a sub-algebra, and where cross relations are expressed using the braiding on $\cA$ and described in Section~\ref{sec:moduliAlgebra}.  We emphasize, however, that in general $a_P$ is \emph{not} simply the braided tensor product of copies of the $\oa$'s -- this appears only in the case of the many-punctured disk -- but rather the relations depend in an interesting way on the pattern $P$: see Section \ref{sec:surfaces}.  
		\begin{theorem}
Given a gluing pattern $P$ for $S$, there is a canonical isomorphism
\[
A_S \cong a_P
\]
as algebras in $\cA$.
		\end{theorem}

The proof is based on applications of the Barr-Beck monadicity theorem developed in Section~\ref{sec:category}, which allow us to the describe various module categories over a rigid tensor category $\cA$ as categories of modules over an algebra internal to $\cA$.  It is really here that the rigidity assumption is most important.

The disk, annulus and once-punctured torus each admit a unique gluing pattern; hence we will denote these simply by $D^2$, $Ann$, $T^2\backslash D^2$, respectively. For higher genus and number of punctures, a given surface may admit several distinct gluing patterns; each such gives a different presentation for the algebra $A_S$.

\subsubsection{The case of quantum groups}\label{qg-defs}
Our main example will come from fixing a reductive algebraic group $G$, and a Killing form $\kappa$ on $\mf g=Lie(G)$.  We will consider the balanced tensor category $\cA = \Repq G$: by this notation we will denote either the category of locally finite-dimensional $U_q(\mf g)$-modules, where $U_q(\mf g)$ is the quantum group associated to $\mf g$ and $\kappa$ when $G$ is simply connected, or more generally the finite-index braided tensor subcategory of $U_q(\mf{g})\modu$ corresponding to $G$ (and determined by its Cartan subgroup) when $G$ is not simply connected.  We do not recall the presentation of $U_q(\mf g)$ here; it can be found e.g. in~\cite[Section~9.1]{Chari1994}.  When $q$ is a root of unity, there are several different versions of $\Repq G$ one may consider, and to which our results apply; see Section \ref{roots-of-one}.  

When $\cA = \Repq G$, the algebras $a_P$ recover several interesting and well-known constructions in the geometric representation theory of quantum groups:
\begin{enumerate}
	\item For the disk, we have $a_{D^2}\cong \mathbf{1}_\cA$, which is just the tautological equivalence $\cA\simeq \mathbf{1}_\cA\modu_\cA$.
	\item For the annulus or cylinder, we have $a_{Ann}=\cO_q(G)$, the reflection equation algebra~\cite{Donin2003a,Donin2003,Donin2002,Donin2002a,Kolb2009,Lyubashenko1995,Lyubashenko1994,Majid1993b}.
\item For the punctured torus, we have $a_{T^2\backslash D^2} \cong \cD_q(G)$, the algebra of quantum differential operators on $G$. This algebra has received a lot of attention in recent years~\cite{Semenov1994,Alekseev1993,Alekseev1996,Backelin2008,Backelin2006,Varagnolo2010,Jordan2009,Jordan2014, Brochier2017}.

\item More generally, the moduli algebras of Alekseev \cite{Alekseev1993}, ~\cite{Alekseev1996} can be recovered as algebras $a_P$ associated to certain gluing patterns for higher genus surfaces.
\end{enumerate}

In particular, it should be emphasized that the algebras $a_P$ have
explicit generators-and-relations presentations, PBW bases, and
interesting representation theory related to the symplectic geometry of
the classical character variety.  These features are not typical of
factorization homology in general, but emerge from the representation
theoretic framework.

On the topological side, the reflection equation algebra, quantum
differential operator algebra, and moduli algebras all admit interesting
actions of mapping class groups of their associated surfaces,
and can moreover be used to produce representations of the
associated surface braid groups, extending Reshetikhin
and Turaev's constructions~\cite{Reshetikhin1990}.  Whereas these
topological structures have historically been themselves constructed via
generators-and-relations comparison, they now follow naturally from the
topological framework of factorization homology.  This is an important
distinction if one wants to produce topological invariants and
categorical structures from these constructions - the categories and
distinguished objects associated to surfaces are functorial, but the
particular presentations are not.

\subsubsection{Quantization of character varieties}

Let $S$ be a surface. There is an affine scheme $\RV_G(S)$ whose $\K$-points are canonically identified with the set of group homomorphisms
\[
	\RV_G(S)=\{\rho:\pi_1(S)\longrightarrow G\}.
\]
Equivalently, $\RV_G(S)$ is the moduli space of local systems equipped with a trivialization at a fixed point of $S$. 

The stack $\uch_G(S)$ is thus the quotient of $\RV_G(S)$ by the action of $G$ changing the trivialization.
It follows that the categories of sheaves on the two spaces recover each other by (de-)equivariantization, and sheaves on $\uch_G(S)$ are given by modules over the algebra of functions on $\RV_G(S)$ when considered as an object in $QC(pt/G)=\Rep G$.

 We denote by $\ch_G(S)$ the $G$-character \emph{variety} (not stack), that is the affine categorical quotient

\[
	\ch_G(S):=\RV_G(S)\dS G
\]
by the natural adjoint action of $G$, to distinguish it from the character stack $\uch_G(S)$. The space $\ch_G(S)$ carries a canonical Poisson structure originally due to Atiyah--Bott~\cite{Atiyah1983} and Goldman~\cite{Goldman1984}. A discrete, combinatorial construction of this structure was given by Fock--Rosly~\cite{Fock1999} using classical $r$-matrices. They construct a Poisson structure on the representation variety $\RV_G(S)$ itself, which depends on the choice of a representation of $S$ as a so-called ciliated graph. When $\cA$ is the category of modules over the formal quantum group $U_{\h}(\mf g)$ and $P$ is a gluing pattern for $S$, the algebra $a_P$ is a flat deformation of the algebra of functions on the representation variety $\RV_G(S)$. By regarding a gluing pattern $P$ for $S$ as a ciliated graph with only one vertex, the Fock-Rosly construction determines a Poisson structure on $\RV_G(S)$.
	
Fock--Rosly's construction was partly inspired by the work of Semenov-Tian-Shansky~\cite{Semenov1994} who introduced a certain dual Poisson structure on $G$, characterized by a classical version of the reflection equation, and a Poisson structure on $G\times G$ (the classical Heisenberg double) thought as a Poisson-Lie version of the cotangent bundle $T^*G$. He also constructed quantization of those structures. The relation between those Poisson structures and $\ch_G(S)$ was already noticed in~\cite{Alekseev1995}.

Quantizations of $\ch_G(S)$ were then obtained in~\cite{Alekseev1993,Alekseev1996} by, roughly, replacing classical $r$-matrices by quantum $R$-matrices, hence giving an FRT-like (in the sense of~\cite{Faddeev1990}) presentation of the sought quantization.  We prove the following:
	\begin{theorem}
		The algebra $a_P$ is a quantization of the Fock--Rosly Poisson structure on $\RV_G(S)$ associated with $P$. Its $U_{\h}(\mf g)$-invariant part is independent of $P$, and is a quantization of the canonical Poisson structure on $\ch_G(S)$.
	\end{theorem}

For suitable choices of gluing pattern $P$, our quantizations recover Alekseev's algebras.   Actions of the mapping class group of the underlying surface on quantizations of character varieties are constructed in~\cite{Alekseev1996a} directly via generators and relations. In our approach this is rather a by-product of their topological definition via factorization homology.

\subsection{Outlook}\label{outlook}
Here we collect several remarks pertaining to further directions of study springing from the current work.

\subsubsection{Computations in closed surfaces}
In the tandem paper \cite{Ben-Zvi2016}, we extend the present techniques to the setting of closed, and possibly marked, surfaces.  The key technical difficulty there is that, while punctured surfaces can be glued from disks along boundary intervals, closed surfaces require gluing along boundary annuli as well. In particular, in order to understand the algebraic data involved in sealing up punctures, or gluing in marked points, we need a convenient monadic framework for working with module categories over the factorization homology, $\int_{Ann}\cA$, of the annulus.  This is accomplished via the theory of quantum moment maps and quantum Hamiltonian reduction.

\subsubsection{Four-dimensional topological field theory}\label{TFT section}
Factorization homology fits naturally into the language of extended topological field theory and the cobordism hypothesis~\cite{Lurie2009, Scheimbauer2014}. This connection is motivational, but not technically necessary for our paper, so we will be informal in our treatment.

One well-studied source of interactions between braided tensor categories and topological field theory is provided by the identification of \emph{modular} categories -- braided tensor categories with strong finiteness -- with the possible values of extended (1,2,3)-dimensional topological field theories on the circle.  This is the Witten-Reshetikhin-Turaev (WRT) construction, originating in Chern-Simons theory.

Our work relates to a different chain of ideas, which braided tensor categories to \emph{four-dimensional} topological field theories such as the topologically twisted ${\mathcal N}=4$ super Yang-Mills theory used by Kapustin and Witten to study the geometric Langlands correspondence. In this story braided tensor categories appear as the possible values of a fully extended four-dimensional field theory on a point. 
The simplest example of such a four-dimensional theory is the four-dimensional ``anomaly theory" for the WRT theory, which was first introduced by Crane-Yetter and Kauffman ~\cite{Crane1997} (and hence will be here-after abbreviated CYK), using a modular category construction. The CYK theory is revisited in recent work of Freed and Teleman~\cite{Freeda}, who show that a modular tensor category $\cM$ defines a fully extended (in fact invertible) 4-dimensional topological field theory valued in the Morita 4-category of braided tensor categories.

The Cobordism Hypothesis~\cite{Lurie2009} establishes that fully extended $n$-dimensional topological field theories are functorially determined by the invariant assigned to an $n$-disk, an object of a higher category with suitable finiteness conditions. 
Following Lurie, Scheimbauer and Haugseng~\cite{Lurie2009, Lurie, haugseng2017, Scheimbauer2014}, an important and accessible special case is provided by the Morita theory of $E_n$-algebras (see also~\cite{Johnson-Freyd2017}).  This yields a higher category whose objects are algebras over the little $n$-disks operad; roughly speaking, the higher morphisms in each dimension $k$ are $E_{n-k}$-algebras equipped with compatible $E_{n-k+1}$ actions of the source and target $k-1$-morphism.  The Morita theory of a given $E_n$-algebra automatically satisfies the finiteness conditions necessary to define invariants of manifolds of dimension at most $n$. The invariants are precisely those given by factorization homology.

Our construction fits in this formalism, since we regard braided tensor categories as $E_2$-algebras in $\LFP$.  In order to connect to the CYK and WRT theories, we may take the complex parameter $q$ appearing in the definition of $\Repq G$ to be a root of unity; all our constructions still hold in this generality.  In this case $\Repq G$ has the modular tensor category appearing in WRT and CYK theories as a semi-simple sub quotient.

It is expected by experts (cf.~\cite{Walker}) that under mild conditions ribbon categories define ``(3+$\epsilon$)-dimensional TFTs", meaning that they have all the lower-dimensional structures of a four-dimensional theory but are not defined on 4-manifolds (as follows for example from the infinite dimensionality of the vector spaces such a theory should attach to certain 3-manifolds). This involves showing that a rigid braided tensor category $\cA$ defines a 3-dualizable object in the Morita 4-category of $E_2$-algebras in $\Rex$, and that a ribbon structure gives rise to a homotopy fixed point structure for the induced $SO(3)$-action.  This is the subject of work in progress with N. Snyder~\cite{Brochier}.

The factorization homology of a balanced tensor category $\cA$ over surfaces forms a category-valued 2-dimensional oriented topological field theory, which produces the same kind of data as the (0,1,2)-dimensional part of an oriented 4-dimensional topological field theory, but formulated in the language of Morita theory. 
To a closed surface we attach the category $\int_S\cA$. 

To a 1-manifold $M$ we attach the monoidal category $\int_{M\times \RR}\cA$, which is a stand-in in the Morita theory for its 2-category $\int_{M\times \RR}\cA\modu$ of modules.
Thus for a surface with collared boundary we have an object in the 2-category attached to the boundary: 
$$\int_S\cA\in \left( \int_{{\partial S}\times \RR}\cA\right)\modu$$ and the sewing property above provides the composition structure for the topological field theory. More generally we obtain invariants for surfaces with boundary components marked by module categories for $\int_{S^1\times \RR}\cA$. 
Finally to a point we attach the $E_2$-category $\cA$ itself, which is the Morita avatar of the 3-category
$(\cA\modu)\modu$ of $\cA$-linear 2-categories (module categories for the monoidal 2-category of $\cA$-module categories).

\subsubsection{Roots of unity}\label{roots-of-one}
The techniques developed in this paper apply to arbitrary rigid braided tensor categories.  In particular, when we consider $\Repq G$, for $q$ a root of unity, there are four distinct settings in which we can work.  The De Concini-Kac quantum group $U^{DK}_q(\mathfrak{g})$ arises from directly specializing $q$ in Serre's presentation for the quantum group~\cite{De1990}.  The resulting algebra has a large center, over which it decomposes (\'etale-locally) as a direct sum of matrix algebras.  The small quantum group, $u_q\mathfrak{g}$ appears as a quotient of $U^{DK}_q(\mathfrak{g})$ at a certain central character~\cite{Lusztig1990}; while $u_q(\mathfrak{g})$ does not admit an $R$-matrix~\cite{Lentner2015}, certain finite-degree extensions of it do.  There is Lusztig's restricted quantum group $U_q^{res}(\mathfrak{g})$, which includes divided powers of Serre generators $E_i$ and $F_i$, whose category of representations give a braided tensor category.  Finally, there is the braided tensor subcategory of tilting modules~\cite{Andersen1991} for $U_q^{res}(\mathfrak{g})$, and its quotient by negligible morphisms, the modular tensor category $\mathcal{M}_q$ of Reshetikhin-Turaev theory~\cite{Reshetikhin1991}.  The four-dimensional TFT defined by $\mathcal{M}_q$ has already been studied by Crane-Kauffman-Yetter~\cite{Crane1997}, but the others appear not to have received the same attention.  It should be very interesting to compare the constructions in the present paper in each of the different root of unity settings outlined above.

\subsubsection{Betti Geometric Langlands}\label{Betti section}

An important motivation for this series of papers is the formulation of a Betti form of the quantum geometric Langlands conjecture. The spectral side of the ``de Rham" Geometric Langlands Conjecture~\cite{Arinkin2014} involves coherent sheaves on the space of flat $G$-connections on an algebraic curve $X$ (i.e., $G$-bundles on the de Rham space $X_{dR}$ of $X$). The Betti version of the conjecture~\cite{Ben-Zvi2016a} replaces this category by coherent sheaves on the character variety $\uch_G(S)$ of the underlying topological surface (i.e., $G$-bundles on the Betti version $X_{Betti}=S$ of $X$).  

The quantization of the de Rham category is constructed using representations of affine Kac-Moody algebras -- in particular through the localization functor studied by Beilinson-Drinfeld from $LG_+$-integrable $\mathfrak g$-modules to twisted $\cD$-modules on $Bun_G(X)$ (depending on a point $x\in X$), where the twisting (or level) plays the role of inverse quantization parameter. 
Likewise the quantization of the Betti categories is constructed in this paper by assembling the localization functors $\Delta_x$ ($x\in S$) from representations of quantum groups. 
The quantum analog of the Riemann-Hilbert correspondence is provided by the Kazhdan-Lusztig equivalence between the two representation categories. We summarize the situation in the following diagram, where dotted arrows indicate equivalences after analytification, and where in the right column we indicate the underlying spaces rather than their categories of sheaves:

$$\xymatrix{
\Repq  G \ar[d]^-\sim_-{KL}  \ar[r]^-{\Delta_x} & \int_S \Repq  G \ar@{.>}[d]_-{q-RH}\ar@{~>}[r]^-{q\to 1} & \uch_G(S) \ar@{.>}[d]_-{RH}    \\
(\widehat{\mathfrak g}\modu_k)^{LG_+} \ar[r]^-{\Delta_x} & \cD_k(Bun_G(X)) \ar@{~>}[r]^-{k\to\infty} & Conn_G(X)
}
$$

The quantum Betti conjecture~\cite{Ben-Zvi2016a} relates the categories constructed in this paper with their automorphic counterparts, which are given by twisted sheaves with nilpotent singular support on $Bun_{G^\vee}(X)$. This conjecture is motivated in turn by the work of Kapustin-Witten~\cite{Kapustin2007}, in which Langlands duality is related to electric-magnetic S-duality in 4-dimensional topological field theory (${\mathcal N}=4$ supersymmetric Yang-Mills theory in the GL twist). The constructions in this paper using factorization homology is an algebraic model for this topological field theory (in contrast with the de Rham version which does not form a topological field theory, but rather depends on the complex structure of the curve).

S-duality is also expected to have another analytic (though not algebraic) manifestation, as a duality between quantum character varieties for Langlands dual groups at dual levels (roughly $q=e^{\pi i k}$ and $q^\vee=e^{-\pi i/k}$). This is a manifestation of the celebrated but mysterious modular invariance in representation theory of quantum groups:  certain aspects of the representation theory of $U_q\mf{g}$ depend naturally not on $q$ but only on the corresponding elliptic curve $\CC/{q^\ZZ}$. This modularity is expressed by Faddeev's modular double~\cite{Faddeev2000} and many subsequent works (see for example~\cite{Teschner2014,Frenkel2014} and references therein) and the Langlands duality for quantized cluster varieties of~\cite{Fock2009}.

\subsubsection{The KZ category}
The Poisson structure on character varieties can also be obtained by an appropriate reduction of a certain \emph{quasi}-Poisson structure on $\RV_G(S)$ introduced in~\cite{Alekseev2002}. This construction was extended and somewhat simplified in~\cite{Li-Bland2015} and quantized by the same authors in~\cite{Li-Bland2014}. The quantization of such a quasi-Poisson structure is an algebra internal to the so-called Drinfeld braided tensor category, which deforms $\Rep G$ as a braided tensor category, using a Drinfeld associator (see \cite{Drinfeld1990a,Drinfeld1989}, or ~\cite{Bakalov2001} for an exposition). An important theorem of Drinfeld~\cite{Drinfeld1990a,Drinfeld1989}, inspired by an earlier result of Kohno~\cite{Kohno1987b}, asserts that there exists a non-explicit equivalence between the Drinfeld braided tensor category and $\Repq G$, for $q=\exp(\h)$. While the introduction of associators makes the quantizations of \cite{Li-Bland2014} somewhat less explicit than the moduli algebras of \cite{Alekseev1993}, the role of braided tensor categories is made conceptually clearer.  We note in particular that the ``fusion'' procedure of \cite{Li-Bland2014} is very similar to the construction of $a_P$ in Section~\ref{sec:surfaces}, from copies of $\oa$.

We thus conjecture that their construction agrees with ours when $\cA$ is the Drinfeld category. This would imply in particular that the algebras they obtain are related to those of Alekseev--Grosse--Schomerus by the Kohno--Drinfeld equivalence. We also expect the representations of surface braid groups obtained from factorization homology with coefficients in the Drinfeld category to coincide with those coming from the monodromy of the KZB equations~\cite{Calaque2009,Enriquez2014}. This would give an explicit computation of the latter using quantum groups.

\subsection{Acknowledgments}

We would like to thank Jacob Lurie for first suggesting the $\Repq G$ TFT as a Betti setting for quantum geometric Langlands and for patiently answering our many questions on Grothendieck categories; Pavel Etingof and Benjamin Enriquez for sharing their ideas concerning elliptic structures on categories; Dan Freed and Kevin Walker for sharing their understanding of topological field theory and braided tensor categories and their closely related works in progress with Teleman and Vazirani, respectively.
We'd also like to thank Martin Brandenburg, David Carchedi, Steve Lack, Mike Shulman, Theo Johnson-Freyd and Ross Street, and especially Daniel Sch\"appi and Chris Schommer-Pries for many helpful comments on 2-categorical issues; John Francis, Greg Ginot, Owen Gwilliam, Claudia Scheimbauer, and Noah Snyder for discussions on factorization homology; and Anatoly Preygel for a discussion of t-structures.
We would like to thank Sam Gunningham, David Nadler, and Monica Vazirani for countless discussions surrounding geometric Langlands and related topics.  Finally, we thank the anonymous referee for many helpful suggestions and questions which have greatly improved the exposition.

The work of DBZ was partly supported by NSF grant DMS-1103525.  The work of DJ was supported by NSF grant DMS-1103778 and by the European Research Council (ERC) under the European Union's Horizon 2020 research and innovation programme (grant agreement no. 637618). DBZ and DJ would like to acknowledge that part of the work was carried out at MSRI, Grant No. DMS-1440140, as part of the Fall 2014 program on Geometric Representation Theory.

\section{Factorization homology}\label{sec:facto}
We review factorization homology following~\cite{Ayala2017,Ayala2015} to which we refer for details; see also~\cite{Ginot2015} for a survey of the theory. 
In this paper we focus on factorization homology over compact surfaces with boundaries. However we will also need to compute factorization homology on simple examples of \emph{manifolds with corners}, i.e. 2-dimensional (paracompact Hausdorff) topological spaces locally modeled on $\RR_{\geq 0}^k \times \RR^{2-k}$. Those are particular examples of stratified manifolds, on which factorization homology is well-defined thanks to~\cite{Ayala2017}. 

\subsection{Basic definitions}

Let $\operatorname{Mfld}^2_{fr}$ (resp. $\operatorname{Mfld}^2_{or}$) be the $(\infty,1)$ category associated to the topological category whose objects are framed (resp. oriented) 2-dimensional manifolds with corners, and morphisms between manifolds $M,N$ is the space $Emb(M,N)$ of framed (resp. oriented) embeddings equipped with the compact-open topology. Hence 1-morphisms are smooth embeddings, 2-morphisms are paths between embeddings (i.e. isotopies), 3-morphisms are homotopies between those and so on. The disjoint union turns those categories into symmetric monoidal categories. Let $\operatorname{Disk}^2_{\partial, fr}$ (resp. $\operatorname{Disk}^2_{\partial, or})$ denote the full subcategory generated under disjoint union by the disks $\RR^k\times \RR_{\geq 0}^{2-k}$, $0\leq k \leq 2$ equipped with their standard framing (resp. orientation). We define similarly $Disk_2^{fr}$ and $\operatorname{Disk}^2_{or}$ as the categories generated by $\RR^2$. Fix an $(\infty,1)$-symmetric monoidal category $(\cC,\bt)$.
\begin{definition}
	A $\operatorname{Disk}^2_{s,B}$-algebra in $\cC$, for $s\in \{fr,or\}$, $B\in \{\emptyset,\partial\}$, is a symmetric monoidal functor from $\operatorname{Disk}^2_{s,B}$ to $\cC$. 
\end{definition}

\begin{remark}\label{rmk:MeqA} One can show, following \cite[Example~5.2.2.15]{Lurie} and~\cite[Proposition~2.12]{Ayala2017}, that the notion of a $\operatorname{Disk}^2_{fr}$-algebra (or rather the image of $\RR^2$) coincides with what is usually called an $E_2$-algebra, or algebra over the little disk operad. Similarly a $\operatorname{Disk}^2_{or}$-algebra is an algebra over the framed little disk operad. A $\operatorname{Disk}^2_{fr,\partial}$-algebra is equivalent to the data of a triple of objects $(\cA,\cM,\cN)$, the image of the triple
	\[
		(\RR^{2},\RR\times \RR_{\geq 0},\RR_{\geq 0}^2),
	\]
	where $\cA$ is an $E_2$-algebra and $\cM,\cN$ are $\cA$-modules satisfying several conditions (e.g. the pair $(\cA,\cM)$ is an algebra over the Swiss-Cheese operad~\cite{Voronov1999}). In this paper we focus to the particular case $\cM=\cN=\cA$, i.e. all disks are sent to the same category regardless of their manifold with corner structure.
\end{remark}

We now assume following~\cite{Ayala2015}, \cite{Lurie} that $\cC$ 
is cocomplete and that for every $c \in \cC$, the functor $c \ot -$ commutes with small colimits. In this setting the following colimit uniquely defines the factorization homology with coefficients in an $E_2$-algebra $\cA\in \cC$:

\begin{definition}[\cite{Ayala2015,Lurie}]
	Factorization homology with coefficients in an $E_2$ (resp. $Disk_2$) algebra $\cA\in \cC$ is the left Kan extension of the above symmetric monoidal functor with respect to the inclusion
	\[
		\operatorname{Disk}^2_{s,\partial} \longrightarrow \operatorname{Mfld}^2_{s}
	\]
	where $s \in \{fr,or\}$, which will be denoted by
	\[
M \longmapsto \int_M \cA.
	\]
\end{definition}

	\begin{definition}[\cite{Ayala2015}]
		A \emph{collar-gluing} for a framed or oriented manifold $M$ is a continuous map
		\[
			f:M \longrightarrow [-1,1]
		\]
		to the closed interval whose restriction to $(-1,1)$ is a manifold bundle. We will write
		\[
			M \cong X^+ \underset{Y \times \RR}{\sqcup} X^-
		\]
		where $X^+=f^{-1}\left([-1,1)\right)$, $X^-=f^{-1}\left((-1,1]\right)$ and $Y=f^{-1}({0})$.
	\end{definition}

One of the main properties of factorization homology is the following:
\begin{theorem}[\cite{Ayala2015,Ayala2017}]\label{thm:excision}
	Let $\cA$ be an $E_2$ (resp. framed $E_2$) algebra in $\cC$. Then the functor $\int_{(-)}\cA$ satisfies, and is characterized by, the following properties:
	\begin{itemize}
		\item If $U$ is contractible, then there is an equivalence in $\cC$
\[
	\int_U \cA \simeq \cA.
\]
\item  A homeomorphism $M\cong Y\times \RR$, for a 1-dimensional manifold with corners $Y$, equips $\int_{M}\cA$ with a canonical $E_1$ structure (from inclusions of intervals inside a larger interval), and any two homeomorphisms induce equivalent $E_1$-algebra structures.
\item
	Let $M=X^- \underset{Y\times \RR}{\sqcup} X^+$ be a collar gluing of $X^{\pm}$ along a codimension 1 sub-manifold $Y$ with a choice of a trivialization $N\cong Y\times \RR$ of a tubular neighborhood $N$ of $Y$. Then $\int_{X^{\pm}} \cA$ are left and right $\int_{Y \times \RR}\cA$-modules respectively, and there is an equivalence of $E_0$ objects in $\cC$
	\[
		\int_M \cA \simeq \int_{X^-} \cA \underset{\int_{Y\times\RR} \cA}{\boxtimes}  \int_{X^+} \cA.
	\]
	\end{itemize}
\end{theorem}

\begin{remark}
	Along the lines of Remark~\ref{rmk:MeqA}, in the particular case considered in this paper, $\int_M \cA$ depends only on the manifold with boundary underlying $M$, not of its manifold with corners structure. This will allow us to see surfaces with circle boundaries as being obtained by gluing closed disks whose boundary is divided into several intervals, and then smoothing the remaining corners. 
\end{remark}
\begin{remark}
Let $S$ be a surface with, say, a single, circular boundary, and let $\tilde S$ be the manifold with corners obtained by subdividing the boundary of $S$ into $n$ intervals. Then, providing that the framings match, or that $\cA$ is a framed $E_2$-algebra, $\int_{\tilde S}\cA$ has a natural structure of a module over
	\[
		\int_{\sqcup_{n} \RR^2}\cA \simeq \cA^{\boxtimes n}.
	\]
	On the other hand, $\int_S \cA$ is naturally a $\int_{S^1\times \RR}\cA$-module. By the previous remark, we have an equivalence 

\[
	\int_{S}\cA \simeq \int_{\tilde S} \cA,
\]
hence the marking induces an $\cA^{\boxtimes n}$-module on $\int_S \cA$ as well. It is easily seen that this module structure is isomorphic to the one obtained via the composition
\[
	\cA^{\boxtimes n}\longrightarrow \cA \longrightarrow \int_{S^1\times \RR} \cA.
\]
\end{remark}

\subsection{Pointed Structure}\label{sec:pointed}

An important additional feature of factorization homology is that it is a \emph{pointed} theory:  the invariant assigned to an $n$-manifold $M$ by an $E_n$-algebra comes equipped with a canonical $E_0$-structure, i.e. a morphism from the unit of the target category $\cC$, constructed as follows.  Let $\emptyset$ be the empty manifold. As $\emptyset$ is the unit for the disjoint union, we have $\int_\emptyset\cA = \un_{\cC}$.  Moreover, $\emptyset$ is an initial object in $\operatorname{Mfld}^2_{or/fr}$; for any manifold $M$, there is a unique embedding:
\[
\emptyset \longrightarrow M.
\]
Its image through factorization homology provides a distinguished morphism
\[
	\un_{\cC} \longrightarrow \int_M \cA.
\]
In the case that $\cC$ is itself some collection of linear categories (as in the next section), we have that $\int_M\cA$ is a category, $\mathbf{1}_{\cC}$ is a category of vector spaces, so that this pointed structure produces a distinguished object in it, denoted $\cO_{\cA,M}$, the image of the one-dimensional vector space. Note that the excision property being an equivalence of $E_0$-objects means that it is compatible with the distinguished object, in the sense that if $M=X^+ \sqcup_{Y\times \RR}X^-$ is a collar-gluing then there is a canonical isomorphism
	\[
		\cO_{\cA,M}\cong\cO_{\cA,X^+}\bt_{\int_{Y\times\RR}\cA}\cO_{\cA,X^-}
	\]
in $\cA$.
	Distinguished objects play a key role in our main result, so we will detail some of their properties in Section~\ref{sec:dist}.

\section{Categorical settings}\label{sec:category}
We now describe the categorical setting in which we will operate.   It is important for constructions such as factorization homology to work in an ambient category (or $\infty$-category) with enough colimits and a symmetric monoidal structure preserving those colimits.   We will largely work in the ``discrete" setting of $1$-categories, so that only a passing knowledge of $\oo$-categories is necessary to read the paper.  However, we must address the existence of co-limits in various 2-categories of categories.  We then recall the basic notions of tensor categories, braided tensor categories, module categories, etc. 

For application in representation theory, we moreover require the notion of $\K$-linear categories.    A category is \emph{$\K$-linear} if it is enriched over the category of $\K$-vector spaces (note there are no finiteness assumptions on Homs or lengths of objects). We will always work with categories enriched over some field $\K$, usually $\K=\CC$, and without further comment except when technically necessarily (though many constructions work with suitable care over more general rings).

It is known that the familiar setting $\Ab$ of small $\K$-linear abelian categories does not possess these features needed for factorization homology, since the Deligne tensor product of two abelian categories is no longer abelian in general~\cite{Franco2013}.  In this section we will recall four convenient categorical settings where this framework is well-defined, and discuss the compatibility between them.  Standard references for this section are~\cite{Aganagic2011,Blumberg2013,Kelly1982,Makkai1989}. We will consider the following (2,1)-categories:

\begin{itemize}
\item $\Rex$, of finitely co-complete $\K$-linear categories with right exact functors,
\item $\Pres$, of presentable $\K$-linear categories with cocontinuous functors,
\item $\LFP$, of compactly generated presentable categories with compact and cocontinuous functors (also known as locally finitely presentable, $\mathbf{LFP}_\K$),
\item $\Gr$, of Grothendieck $\K$-linear abelian categories which admit exact filtered colimits and a small generator (these are in particular presentable), see Remark~\ref{derived}.
\end{itemize}

The 2-morphisms in each case are the $\K$-linear natural isomorphisms.  We note in passing that since each class of categories is defined by requiring closure under colimits of a certain shape, and since colimits in functor categories are computed pointwise in the target, it follows that each of these $2$-categories admits internal Homs, i.e. the Hom between two categories in $\Rex$ (resp. $\Pres$, $\LFP$, $\Gr$) is again a category in $\Rex$ (resp. $\Pres$, $\LFP$, $\Gr$).

As motivating examples:  for a $\K$-algebra $A$, the categories $A\modu_{fd}$, of $A$-modules of finite-dimension over $\K$, and $A\modu_{fp}$, of finitely presented $A$-modules, both live in $\Rex$.  On the other hand, the categories $A\modu_{lf}$, of $A$-modules which are locally finite dimensional (i.e. every vector generates a finite-dimensional submodule), and $A\modu$, of all left $A$-modules, both live in $\LFP$ and in fact in $\Gr$.

{\bf Terminology.} By a {\em{2-category}} we will always mean a (non-strict) $(2,1)$-category, in other words, we will not use any non-invertible 2-morphisms. Equivalently, this means an $(\infty,1)$-category with 1-truncated mapping spaces.  In particular, the notion of colimits in the $(\infty,1)$ setting coincides with the 2-categorical notion, and with the notion of a homotopy colimit of categories (See Section 4.2.4 of \cite{Lurie2009a}).  We will therefore use these notions interchangeably.

\subsection{An equivalence between $\Rex$ and $\LFP$}
Recall that a category is presentable (also known as locally presentable) if it is accessible (generated under colimits by a small subcategory) and cocomplete (closed under small colimits -- in fact this implies it is complete as well). We denote by $\Pres$ the 2-category of presentable $\K$-linear categories with colimit preserving (aka cocontinuous) functors and natural isomorphisms.

Presentable categories provide a very flexible setting for algebra:  tensor products of presentable categories are again presentable; so are colimits of presentable categories; the adjoint functor theorem provides right adjoints to colimit preserving functors, to which we can then apply techniques such as Barr-Beck monadicity to describe categories explicitly as categories of (co)modules. In practice the presentable categories we encounter on punctured surfaces will all be (Grothendieck) abelian categories. (See also Remark~\ref{derived}.)

  Recall that an object $c\in \cC$ is compact (resp. compact projective) if $Hom(c,-)$ preserves filtered colimits (resp. arbitrary colimits), and compact-projective if it preserves arbitrary colimits.  We denote by $\cC_c$ the full subcategory of compact objects.  The ind-completion $\ind(\cC)$ of a small category $\cC$ is the universal category containing all filtered colimits of diagrams in $\cC$.  It may be constructed as a category of filtered diagrams in $\cC$.

A category $\cC$ is said to be compactly generated if any object is a filtered colimit of compact objects (equivalently, we can identify $\cC$ with the ind-completion of its full subcategory of compact objects).

We denote by $\LFP\subset \Pres$ the subcategory of compactly generated $\K$-linear presentable categories and {\em compact} functors, i.e., functors preserving compact objects or equivalently (by the adjoint functor theorem) possessing right adjoints preserving filtered colimits.

\begin{remark} We will adopt the usual convention of referring to objects in $\ind\cC$ for a  category $\cC$ as ind-objects in $\cC$.  For instance, an ind-algebra in a monoidal category $\cC$ really means an algebra object in $\ind\cC$ with respect to the co-continuous extension of the monoidal structure.
\end{remark}

A category is \emph{finitely co-complete} if it admits finite colimits. We let $\Rex$ denote the 2-category of essentially small, finitely cocomplete categories with morphisms \emph{right exact} functors (i.e. functors preserving finite co-limits) and natural isomorphisms. (Recall that a category is essentially small if it is equivalent to a small category.)
Note that since abelian categories are finitely cocomplete, there's a full subcategory of $\Rex$ consisting of small abelian categories, with right exact functors.  

Given a $\Rex$ category $\cC$, we may consider its ind-completion $\ind(\cC)\in \Pres$, which will be $\LFP$ by construction. Given a $\Rex$ functor $F:\cC\to\cD$, we may consider the ind-extension $\ind{F}:\ind(\cC)\to\ind(\cD)$, which will be co-continuous and compact, i.e., will preserve the subcategory of compact objects. Conversely, given a presentable category $\cC\in \Pres$, we may consider its subcategory $\comp(\cC)$ of compact objects, an object of $\Rex$. Restricting then to compactly generated categories and compact functors, we find the operations $\ind$ and $\comp$ define an equivalence of (2,1)-categories,$$\xymatrix{\ind:\Rex \ar[r]^-\sim &\ar[l] \LFP:\comp}.$$

It is convenient to go back and forth between the concrete setting of small categories $\Rex$ and the flexible setting of presentable ones $\Pres$: the monadic techniques we will develop take $\Rex$ categories as output, but most naturally produce concrete descriptions of their ind-completions in $\Pres$. Thanks to the above equivalence, we will move between the settings interchangeably.

We also record here the following proposition, which was explained to us by Daniel Sch\"appi:

\begin{proposition}\label{Schaeppi}: If $\cC$ is an abelian category and $T$ is a right exact monad on $\cC$, then $\cD:=T-mod_\cC$ is abelian. 
\end{proposition}

\begin{proof} The forgetful functor $\cD\to\cC$ creates all limits that exist in $\cC$ (for any monad), in particular finite ones, and all finite colimits (since $T$ preserves finite colimits). Since the forgetful functor is also conservative and $\cC$ is abelian, it follows that the comparison morphism 
$$coker(ker(f)) \longrightarrow ker(coker(f))$$
is an isomorphism in $\cD$, hence $\cD$ is abelian.
\end{proof}

\subsection{The Deligne-Kelly tensor product}

For $\K$-linear categories $\cC, \cD, \cE$, let $\operatorname{Bilin}(\cC\times \cD,\cE)$ denote the category of $\K$-bilinear functors from $\cC\times \cD$ to $\cE$, preserving finite colimits in each variable separately, and with natural isomorphisms as morphisms.

\begin{definition} The \emph{Deligne-Kelly tensor product} $\cC\bt\cD$ of $\cC, \cD \in \Rex$ is 
uniquely characterized by the natural equivalences,
$$\Rex[\cC\bt \cD,\cE] \simeq \operatorname{Bilin}(\cC\times \cD,\cE).$$
\end{definition}
In~\cite{Kelly1989}, it is shown that Kelly tensor product equips $\Rex$ with the structure of a symmetric closed monoidal (2,1)-category, in particular we have a further equivalence:
$$\Rex[\cC\bt \cD,\cE] \simeq \Rex[\cC,\Rex[\cD,\cE]].$$
We denote this symmetric closed monoidal $(2,1)$-category by $\Rex^\bt$.

The Kelly tensor product extends to $\LFP$, where it is characterized by the analogous universal property (with respect to functors preserving colimits in both factors), and the functor $\ind$ extends to an equivalence $\Rex^\bt\simeq\LFP^\bt$ of symmetric monoidal $(2,1)$-categories.

\begin{remark} It is shown in~\cite{Franco2013} that the Deligne tensor product of abelian categories -- when it exists -- coincides with the Kelly tensor product, but that the former may not exist in general.  This, essentially, is the reason to work with $\Rex$ and $\LFP$, rather than $\Ab$. 
\end{remark}

\begin{proposition}\label{excision} The symmetric monoidal $\infty$-category $\Rex^\bt\simeq\LFP^\bt$ is closed under small $2$-colimits, and the tensor product preserves $2$-colimits in each factor separately.
\end{proposition}
\begin{proof}
In~\cite{Blackwell1989} it is shown that $\Rex$ (equivalently, $\Lex$) is the category of 2-modules in $\operatorname{\mathbf{Cat}}$ of a finitary 2-monad $T$, and as such is closed under arbitrary bicolimits. Because $\Rex^\bt$ is \emph{closed} monoidal, the functor $\cC\bt -$ has a right adjoint, and therefore commutes with arbitrary colimits. (Recall from Section~\ref{sec:category} that colimits in the 2-categorical setting are identified with their $\infty$-categorical versions.)
\end{proof}

In particular it follows that $\Rex$ (and hence $\LFP$, by the equivalence $\Rex\simeq\LFP$ of symmetric monoidal $(2,1)$-categories) satisfies the conditions (*) of~\cite{Ayala2015} for the definition of factorization homology (for which a much smaller class of colimits is required, namely the sifted ones).

\subsection{Tensor and braided tensor categories}\label{sec:tensorcat}

It is well known that $E_1$-, $E_2$- and framed $E_2$-algebras, respectively, in $\cat^\times$ are equivalent to monoidal, braided monoidal and balanced braided monoidal categories (see~\cite[Example 5.1.2.4]{Lurie} and~\cite[Chap. 6]{Fresse2017}).  In this section we consider the $\K$-linear analogs of these structures, as well as the corresponding notion of rigidity (for a general introduction to rigid tensor categories, see e.g.~\cite{Etingof2015}).

\begin{definition} A \emph{tensor category} in $\LFP$ is an $E_1$-algebra $\cA$ in $\LFP$.  Similarly a \emph{braided tensor} category in $\LFP$ is an $E_2$-algebra in $\LFP$.\end{definition}

This definition is a compact formulation of the traditional notions of tensor and braided tensor categories. Since $\LFP$ is a $2$-category, the data of an $E_1$-algebra consists of a category, a functor of tensor product, and an associator natural isomorphism, which satisfies the so-called ``pentagon equation".  We will follow the usual convention of dropping explicit mention of associators in formulas, for clarity of exposition.  We will denote by $\cA^{\ot-op}$ the tensor category $\cA$ with opposite tensor product, so $V\ot^{op} W = W\ot V$.

\begin{definition}\label{def:rigid} A tensor category $\cA$ in $\LFP$ is \emph{rigid} if all compact objects of $\cA$ are left and right dualizable.  A braided tensor category is rigid if its underlying tensor category is rigid.\end{definition}

For later use, we note here that the braiding on a braided tensor category $\cA$ endows the iterated tensor functors,
$$T^k: \cA^{\bt k}\to \cA,$$

$$a_1\bt\cdots \bt a_k \to a_1\ot\cdots\ot a_k,$$
with the structure of a tensor functor, via the following ``shuffle" braiding:
\begin{align}\label{eq:tensor}
J_{\mathbf{a},\mathbf{b}}: a_1\ot\cdots\ot a_k \ot b_1\ot \cdots \ot b_k \xrightarrow{\simeq} a_1\ot b_1\ot \cdots\ot a_k\ot b_k
\end{align}
$$J_{\mathbf{a},\mathbf{b}} = \sigma_{a_k,b_{k-1}}\circ\cdots\sigma_{a_3\ot\cdots\ot a_k,b_2}\circ\sigma_{a_2\ot\cdots\ot a_k,b_1}.$$

\begin{remark}\label{rmk:balanced} In order to define a $Disk^2_{or}$-algebra in $\Rex$, we require $\cA$ to be equipped with a balancing~\cite{Salvatore2001} i.e. an automorphism $\theta$ of the identity functor on $\cA$, satisfying
$$\theta_{V\ot W} = \sigma_{W,V}\sigma_{V,W}(\theta_V\ot \theta_W).$$
It is well-known~\cite{Selinger2011} that, having already assumed $\cA$ is rigid, this is equivalent to equipping $\cA$ with a pivotal structure.
\end{remark}

\begin{definition} Let $\cA$ be a tensor category in $\LFP$.
\begin{enumerate}
\item A (right) $\cA$-module category $\cM$ for a tensor category $\cA$ in $\LFP$ is a category $\cM\in\LFP$, together with an action functor in $\LFP$,

\begin{equation*}
act_\cM: \cM\bt\cA \to\cM,
\end{equation*}
satisfying standard associativity (pentagon) axioms.  We will abbreviate $\act_M(m \bt X)$ by $m \ot X$.

A left module category is defined similarly\footnote{Henceforth, ``module category" will connote ``right module category"}.
\item An $(\cA,\cB)$-bimodule category for tensor categories $\cA$ and $\cB$ is, equivalently, a right module category for $\cA^{\ot-op}\bt\cB$ or a left module category for $\cA\bt\cB^{\ot-op}$.
\item For $m\in\cM$, we denote by $\act_m$ the action functor on $m$,
\begin{align*}\act_m:\cA\to\cM\\
a\mapsto m\ot a.\end{align*}
This functor has a right adjoint,
$$\act_m^R: \cM\to\cA.$$
\item  For any $n\in \cM$, we denote by $\underline{\Hom}_\cA(m,n)$ the \emph{internal homomorphisms} from $m$ to $n$, which by definition is the object $\act_m^R(n)\in \cA$. For any triple of objects $m,n,p \in \cM$ there is a well-defined composition map in $\cA$ (see e.g.~\cite{Etingof2015})
		\[
			\iHom(n,p)\ot \iHom(m,n)\longrightarrow \iHom(m,p).
	\]
\item We denote by $\underline{\End}_\cA(m):=\iHom(m,m)=\act_m^R(\act_m(\mathbf 1))$ the \emph{internal endomorphism algebra of $m$}, which carries a natural algebra structure.
\end{enumerate}
\end{definition}

\begin{remark}
To avoid confusion, let emphasize that it is built into our assumptions for tensor and braided tensor categories, module categories, and their module functors is the requirement that all structural functors are compact-preserving.
\end{remark}

\begin{proposition}\label{prop:adjoint}
Let $\cA$ be a rigid tensor category in $\LFP$, let $\cM$ and $\cN$ be $\cA$-module $\LFP$ categories,
and let $F:\cM\to\cN$ an $\cA$-module $\LFP$ functor admitting a right 
 adjoint $F^R$ as a plain functor. Then $F^R$ admits a canonical $\cA$-linear structure.
\end{proposition}
\begin{proof}
The assertion appears as Corollary 2.13 of \cite{Douglas2014} (and as an exercise to the reader in~\cite[3.3]{Etingof2004}); see also~\cite[Lemma 3.5]{Ben-Zvi2009} in the dg setting. 
 The right adjoint to an $\cA$-module functor automatically commutes with the $\cA$-module structure in a lax sense, i.e., up to natural transformation. The right adjoint also automatically commutes strictly with the action of dualizable objects of $\cA$, which can be moved across $Hom$ pairings. Finally the rigidity of $\cA$ allows us to write objects of $\cA$ as filtered colimits of dualizable objects, whence the Proposition follows.\end{proof}

\begin{proposition}\label{prop:adMult} Let $\cA$ be a rigid tensor category in $\LFP$.  Then the tensor product functor $T:\cA\bt\cA\to\cA$ has a co-continuous right adjoint.
\end{proposition}

\begin{proof}  
	By assumption, $T$ is cocontinous so it admits a right adjoint $T^R$ which is a priori only a linear functor. We need to prove that $T^R$ is cocontinuous if $\cA$ is rigid. Since the tensor product in $\cA$ is associative, $T$ has a canonical $\cA$-bimodule structure, and since $\cA$ is rigid, $T^R$ is an $\cA\boxtimes \cA^{\ot-op}$-module functor. Hence for any $X\in \cA$,
	\[
		T^R(X)=T^R((X\boxtimes \un_\cA)\ot \un_\cA)\cong (X\boxtimes \un_\cA)\ot T^R(\un_\cA).
	\]
	Since the tensor product of $\cA\boxtimes \cA^{\ot-op}$ is cocontinuous in each variable, $T^R$ is cocontinuous.
\end{proof}

\begin{proposition}\label{prop:adAct}
Let $\cA$ be a rigid tensor category, and $\cM$ a right $\cA$-module. Then the action functor
\[
\act:\cM\boxtimes \cA\rightarrow \cM
\]
has a cocontinuous right adjoint.
\end{proposition}
\begin{proof}
The proposition and its proof carry over without modification from the corresponding assertion in the dg category setting,	
Proposition D.2.2 of~\cite{Gaitsgory2015}. Namely one can write down explicitly a cocontinuous right adjoint $\act^R$ and its adjunction data. The functor $\act^R$ is given as a composition
$$\xymatrix{\cM\ar[rr]^-{Id_\cM\boxtimes \un_\cA }&& \cM\boxtimes \cA\ar[rr]^-{ Id_\cM\boxtimes T^R}&&
\cM\boxtimes \cA\boxtimes \cA\ar[rr]^-{\act\boxtimes Id_\cA}&& \cM\boxtimes\cA}$$ of cocontinuous functors, hence is cocontinuous.
\end{proof}

\begin{definition}
The \emph{relative Kelly tensor product} $\cM\bt_\cA\cN$ of left and right module categories $\cM$ and $\cN$ for a tensor category $\cA$ is defined as the colimit of the infinite 2-sided bar construction for $\cM$ and $\cN$ -- i.e., the geometric realization of the simplicial category 
$$
 \xymatrix{\cdots 
 \ar[r]<.75ex> \ar[r]<.25ex> \ar[r]<-.25ex> \ar[r]<-.75ex> &
 \cM\bt\cA\bt\cA\bt\cN\ar[r]<.5ex> \ar[r] \ar[r]<-.5ex> &
 \cM\bt\cA\bt\cN\ar[r]<.25ex> \ar[r]<-.25ex>&
 \cM\bt \cN}
$$
\end{definition}

\begin{remark}  
Thus, we take as the definition of the relative tensor product the standard $\infty$-categorical definition of relative tensor product of left and right modules over an algebra object in a symmetric monoidal $\infty$-category.  
Existence of relative tensor products in $\LFP^\bt$ is an easy consequence of existence of Kelly tensor products, and closure of $\LFP^\bt$ under colimits. It is not hard to show, along the lines of MacLane's coherence theorem, that the resulting colimit is can be calculated much more concretely:  essentially, because $\LFP$ is only a 2-category, the infinite bar construction strictifies after the second step. 

The resulting construction may be characterized as the universal source for {\em balanced} functors out of $\cM\bt\cN$,
$$\LFP[\cM\bt_\cA \cN,\cE] \simeq \operatorname{Bal}_\cA(\cM\bt \cN,\cE).$$
Recall from~\cite{Etingof2010}, Definition 3.1, that a functor $F:\cM \bt \cN\to \cE$ (where $\cM, \cN$ are right and left $\cA$-module categories and $\cE\in\LFP$) is called \emph{$\cA$-balanced} when equipped with natural isomorphisms,
\begin{equation}B_{m,X,n}:F(m\ot X \bt n) \cong F(m \bt X\ot n),\label{eqn:rel-tens}\end{equation}
for $m\in\cM, X\in\cA, n\in\cN,$ satisfying certain coherences.  We denote by $\operatorname{Bal}_\cA(M\bt N, \cE)$ the category of balanced functors.  Often, e.g. in \cite{Etingof2010}, formula \eqref{eqn:rel-tens} is taken as the definition of the relative tensor product $\cM\bt_\cA \cN$.  See~\cite{Douglas2014} for an explicit study of balanced tensor products.

\end{remark}

\begin{remark}
	The relative tensor product of two $E_0$, i.e. pointed, modules is canonically again pointed via
	\[
\Vect\simeq\Vect \bt \Vect \rightarrow \cM \bt \cN \rightarrow \cM\bt_\cA \cN
	\]
	where the first arrow is the absolute tensor product of the pointings of $\cM$ and $\cN$ and the second is the canonical functor.
\end{remark}

\begin{proposition}\label{all-abelian} Let $\cA$ be a rigid tensor category in $\LFP$, and let $\cM,\cN$ be module categories.
Then the right adjoint of the canonical functor $$\mu:\cM\bt\cN \to \cM\bt_\cA\cN$$ is monadic.
\end{proposition}

\begin{proof}

The proof is a variant of the argument of Theorem 4.7 in~\cite{Ben-Zvi2010}, in which one turns colimits of presentable
categories with continuous functors to limits along their right adjoints, using the adjoint functor theorem, specifically Corollary 5.5.3.4 in~\cite{Lurie2009b}. These limits can be calculated as limits of plain categories, since the forgetful functor
from presentable categories (with either left or right adjoints) preserves small limits, see Proposition 5.5.3.13, Theorem 5.5.3.18 in~\cite{Lurie2009b}.

In more detail, recall that  $\cM\bt_\cA \cN$ can be realized as the colimit in $\Pres$ of the simplicial diagram given by the two-sided bar construction, where the simplices are iterated tensor products $\cM\bt\cA \cdots\bt\cA\bt\cN$. Such a colimit of presentable categories with left adjoints
can be calculated as the limit in the opposite $\infty$-category $\Pres^R$ of presentable categories with right adjoints over the corresponding cosimplicial diagram of right adjoints. By Propositions~\ref{prop:adMult} and~\ref{prop:adAct}, all functors appearing in the two-sided bar construction have cocontinuous right adjoints, so that we are in fact calculating a limit over a cosimplicial diagram in $\Pres$. In particular it follows that $\mu^R$ is itself a morphism in $\Pres$, i.e., cocontinuous (as well as continuous by virtue of being a right adjoint).   
The limit in turn is identified with the corresponding limit in the 2-category of categories, i.e., with the category of compatible collections of objects in the cosimplices. 
However forgetting from compatible collections to objects of the 0-th term is a conservative functor: a map of compatible collections (whose objects are determined by the 0-th term) is an isomorphism if and only if it is so on the 0-th term. Thus $\mu^R$ is continuous, cocontinuous and conservative, and thus monadic.
\end{proof}

\begin{corollary} If $\cA,\cM$ and $\cN$ are moreover abelian then so is $\cM\bt_\cA\cN$.\end{corollary}\label{rel-tens-abelian}
\begin{proof}
Since we are working with compact functors, the monad from the proposition is itself colimit preserving. Thus Proposition~\ref{Schaeppi} applies, and when $\cA,\cM$ and $\cN$ are abelian so is $\cM\bt_\cA \cN$.
\end{proof}

\section{Barr-Beck reconstruction of module categories}\label{sec:barr-beck}
\subsection{Barr--Beck Theorem}


A key tool in our computation is the well-known yoga of monads (i.e. of unital algebra objects in categories of endofunctors); in particular, monads arising from adjunctions of tensor functors and module functors.  Many of these results are straightforward extensions to $\LFP$ of the work of Ostrik~\cite{Ostrik2003,Etingof2004}, in the setting of fusion categories, and~\cite{Douglas2013} for general finite abelian categories; we are grateful to Ostrik and Snyder for many explanations.

We begin by recalling Beck's monadicity theorem.  Let $(L,R): \cC \xrightleftharpoons[R]{L} \cD$ be an adjoint pair of functors.  The composition $T=R\circ L$ is naturally a monad on $\cC$ via the adjunction unit and counit,
$$ \eta: \id_\cC\to R\circ L,\qquad m: R\circ L\circ R \circ L \xrightarrow{\id\ot\epsilon\ot\id} R\circ L.$$

\begin{definition} We denote by $T\modu_\cC$ \emph{the category of $T$-modules\footnote{These are sometimes called $T$-algebras, instead of $T$-modules.} in $\cC$}:  objects are pairs $(X, f)$ of an object $X\in\cC$, and a morphism $f:T(X)\to X$; morphisms from $(X,f)$ to $(Y,g)$ are those $h:X\to Y$ making the obvious diagram commute.
\end{definition}

We obtain a functor, $\widetilde{R}: \cD \to T-mod_\cC$, sending $A\in \cD$ to $R(A)\in \cC$, with its canonical $T$-action:
$$act: R\circ L\circ R(A) \xrightarrow{\id\ot\epsilon} R(A).$$

\begin{theorem}[Barr-Beck theorem; see~\cite{MacLane1998} p. 147-150] \label{rexmonad} The functor $\widetilde{R}$ is an equivalence if, and only if:
\begin{itemize}
\item $R$ is conservative/reflects isomorphisms, i.e. if $f:X\to Y$ in $\cD$, is such that $R(f)$ is an isomorphism, then $f$ is an isomorphism.
\item $\cD$ has coequalizers of $R$-split parallel pairs (those parallel pairs of morphisms in $\cD$, which $R$ sends to pairs having a split coequalizer in $\cC$), and $R$ preserves those coequalizers.
\end{itemize}
\end{theorem}

\begin{lemma} Suppose that $\cD$ is abelian.  Then $R$ is conservative if, and only if, for any $X$ with $R(X)\cong 0$ we have $X\cong 0$.\end{lemma}
\begin{proof}
Suppose we have $f:X\to Y$ , and that $R(f)$ an isomorphism.  Since $R$ is right exact (in fact cocontinuous) it preserves finite colimits, and since $R$ is a right adjoint, it preserves arbitrary limits; hence $R(\ker f)=\ker R(f)=0$ and $R(\coker f)=\coker R(f)=0$, hence $\ker f$ and $\coker f$ are zero, hence $f$ is an isomorphism.
\end{proof}

\begin{remark} We note that without the abelian assumption, this only gives us that $f$ is monic and epic. Thus to check conservativity of functors in what follows, it will be extremely useful to know the categories involved are abelian --  so that monic and epic maps are isomorphisms. 
\end{remark}
\subsection{Reconstruction for module categories}

\begin{definition}[See~\cite{Ostrik2003,Douglas2013}]  Let $\cA$ be an abelian tensor category in $\LFP$ with abelian module category $\cM\in\LFP$.  We say that $m\in\cM$ is:
\begin{enumerate}
\item  An \emph{$\cA$-generator}, if $\act_m^R$ is faithful.
\item  \emph{$\cA$-projective} if $\act_m^R$ is colimit-preserving.\footnote{We note that this is equivalent to asking that $\act_m^R$ preserves finite colimits, since by construction $\act_m^R$ preserves filtered colimits.}
\item An \emph{$\cA$-progenerator}, or progenerator for the $\cA$-action, if it is an $\cA$-projective $\cA$-generator.
\end{enumerate}
\end{definition}

\begin{theorem}[Monadicity for module categories]\label{monadicitymodules} Let $\cA$ be a rigid abelian tensor category in $\LFP$, and let $\cM\in\LFP$ be an abelian $\cA$-module category with an $\cA$-progenerator $m\in\cM$.  Then the functor $\widetilde{\act_m^R}$ is an equivalence of $\cA$-module categories,

$$\cM\simeq \underline{\End}(m)\modu_\cA,$$
where $\cA$ acts on the right by multiplication.
\end{theorem}
\begin{proof} We apply Theorem \ref{rexmonad} to the functor $\act_m:\cA\to\cM$.  We need to check that the right adjoint $\act^R_m$ is conservative and co-continuous, which is precisely the assumption that $m$ is a pro-generator.  Hence $\act^R_m$ is monadic.  Because $\cA$ is rigid, $\act^R_m$ carries a canonical module structure, so that the monad $\act^R_m\circ \act_m$ is a module functor, and can therefore be identified with the functor of tensoring with the algebra object $\act^R_m\circ \act_m(\mathbf{1}_\cA)\cong \underline{\End}(m)$.
\end{proof}

\begin{remark} Note that a \emph{right} $\cA$-module category with an $\cA$-progenerator is identified with the category of \emph{left} modules for the internal endomorphism algebra, and vice versa.
\end{remark}

\begin{definition} A functor $F:\cA\to\cB$ is called \emph{dominant} if every object of $\cB$ appears as a sub-object (equivalently using rigidity, quotient) of an object in the image of $F$.\end{definition}

\begin{remark}
Note that, by Lemma 2.1 of~\cite{Bruguieres2011}, a tensor functor $F:\cA\to\cB$ is dominant if, and only if, its right adjoint $F^R$ is faithful, i.e. $\cB$ should be generated under colimits by the image of $\cA$.
\end{remark}

There are two senses in which the construction of internal endomorphism algebras is functorial:  with respect to $\cA$-module functors, and with respect to base change along a dominant tensor functor $F:\cA\to\cB$.

\begin{theorem}[Functoriality of monads]\label{monad-functoriality}
Let $\cM$ and $\cN$ be $\cA$-module categories in $\LFP$, $F:\cM\to\cN$ an $\cA$-module functor, and let $m\in \cM$.  
Then we have a canonical homomorphism of algebras,
$$\rho_F: \underline{\End}_\cA(m) \to \underline{\End}_\cA(F(m)),$$
\end{theorem}
\begin{proof}
We have the following commutative diagram of functors:
$$\xymatrix{
&& \cM\ar@/^/[dll]\ar@/^/[dd]^F\\
\cA\ar@/^/[urr]^{\act_m}\ar@/_/[drr]_{\act_{F(m)}}\\
&&\cN \ar@/_/[ull] \ar@/^/[uu]^{F^R}.
}$$
Hence, we define:
$$\rho_F:\underline{\End}(m) = \act_m^R\circ \act_m(\mathbf{1}) \xrightarrow{\eta_F} \act_m^R\circ F^R\circ F\circ\act_m(\mathbf{1})$$
$$\cong \act_{F(m)}^R\circ \act_{F(m)}(\mathbf{1}) = \underline{\End}(F(m)).$$

It is straightforward to see that $\rho_F$ is compatible with monadic composition, so defines an algebra homomorphism.
\end{proof}

\begin{corollary}  Suppose that $m$ is an $\cA$-progenerator for $\cM$, and $F(m)$ is an $\cA$-progenerator for $\cN$, so that we have equivalences of $\cA$-module categories
$$\cM\simeq \underline{\End}(m)\modu_\cA,\qquad \cN\simeq \underline{\End}(F(m))\modu_\cA.$$
Then $F^R$ is naturally isomorphic to the pull-back functor along $\rho_F$.
\end{corollary}
Recall that if $a$ is an algebra in a tensor category $\cA$ and $\cN$ is a left $\cA$-module category, then by definition for any $n \in \cN$ we have that $a\ot n \in \cN$ so that one can consider the category of $a$-modules in $\cN$, which we denote by $a\modu_\cN$.  This category no longer carries an $\cA$-module structure, in general.

\begin{theorem}[Monadicity for relative tensor products]\label{reltens} Let $\cA$ be a rigid abelian tensor category, and let $\cM$, $\cN$ be right and left $\cA$-module abelian categories, respectively, with $\cA$-progenerators $m\in \cM$ and $n\in\cN$.  Then we have equivalences of categories:
$$\cM \bt_{\cA}\cN \simeq \underline{\End}(m)\modu_\cN \simeq (\underline{\End}(m)\textrm{-}\underline{\End}(n))\operatorname{-bimod}_{\cA}.$$
\end{theorem}
\begin{proof}
The second asserted equivalence is obvious, so we focus on the first.
First we note that the $\cA$-module category $\cM$ is $\cA$-dualizable, i.e., we have a category $\cM^\vee$ together with a coevaluation
$$\Vect\rightarrow \cM\bt_\cA \cM^\vee$$ and an evaluation
$$\cM^\vee\bt \cM\to \cA$$ satisfying the standard duality identities. Indeed
by Theorem~\ref{monadicitymodules}, we have
$$\cM\simeq \underline{\End}(m)\modu_\cA,$$
whence we have a candidate for the dual as
$$\cM^\vee=\modul_\cA-\underline{\End}(m),$$
which is moreover pointed by the $\cA$-module functor $\underline{\Hom}(m,-)$.  The evaluation is given by the relative tensor product of right and left $\underline{\End}(m)$-modules in $\cA$ and the coevaluation given by the image of the pointing (distinguished object) of $\cM\ot \cM^\vee$.
Hence, we have a functor,
$$ G:\cM\bt_\cA\cN\to Fun_\cA(\cM^{\vee},\cN) \xrightarrow{ev_m} \cN$$ 
where the first map is an equivalence, induced by the evaluation $$\cM^\vee \bt \cM\bt_\cA \cN\longrightarrow \cN,$$
and the second map is conservative by the assumption that $m$ is an $\cA$-generator: an $\cA$-functor vanishing on $m$ must vanish identically (here we use the property that $\cM$ is abelian to characterize conservativity of the functor by checking on objects).  Hence the functor $G$ is conservative.  It is easily checked that $n'\mapsto m\bt_\cA n'$ defines its left adjoint $G^L$, and that $GG^L(n')\cong \underline{\End(m)}\ot n'$.  Hence we obtain from Beck monadicity an equivalence,
$$\cM\bt_\cA \cN \simeq \underline{\End(m)}\modu_\cN,$$
as claimed.
\end{proof}

\begin{corollary}[Monadicity for base change]\label{base-change} Let $\cM$ be an abelian $\cA$-module category, $F:\cA\to\cB$ a dominant tensor functor, and $m\in \cM$ a $\cA$-progenerator.  Then $m\bt_\cA\mathbf{1}_\cB$ is a $\cB$-progenerator of $\cM\bt_\cA\cB$, and we have an equivalence of $\cB$-module categories,

$$\cM \bt_{\cA}\cB \simeq F(\underline{\End}(m))\modu_{\cB}.$$
\end{corollary}
\begin{proof}
This is a special case of Theorem \ref{reltens}, simply noting that the functors $G$ and $G_\cA$ are canonically $\cB$-module functors in the case $\cN=\cB$, regarded as a left $\cA$-module through the functor $F$, and right $\cB$-module by the right regular action of $\cB$ on itself.
\end{proof}

\begin{remark} We note that, in the absence of a braiding, the category of $A-B$ bimodules over algebras $A$ and $B$, internal to a tensor category $\cA$, cannot be written as the category of $A\ot B^{op}$-modules in $\cA$, as in the case $\cA=\Vect$, simply because there is no natural algebra structure on $A\ot B^{op}$.
\end{remark}

\subsection{Reconstruction for braided tensor categories}
We shall assume from now on that $\cA$ is a rigid braided tensor category in $\LFP$, which we shall moreover assume to be an abelian category.  In Proposition \ref{prop:1AA} and Theorem \ref{thm:tr-monad} below, we highlight two important special cases where monadicity applies.

First, recall that the braiding of $\cA$ induces a tensor structure on the multiplication map $\cA^{\bt n}\rightarrow \cA$, hence a left and a right $\cA^{\bt n}$-module structure on $\cA$, which we call the left and right regular actions of $\cA^{\bt n}$ on $\cA$. We have:
\begin{proposition}\label{prop:1AA}For any $n$, the tensor unit $\mathbf{1}_\cA$ is a progenerator for the $n$-fold right regular action of $\cA^{\bt n}$ on $\cA$.
\end{proposition}
\begin{proof} The unit isomorphisms $X\cong X\ot\mathbf{1}_\cA$ show that $\act_{\mathbf{1}_\cA}$ is essentially surjective, hence dominant.  All that remains is to show that $\mathbf{1}_\cA$ is $\cA$-projective, and this is the content of Proposition \ref{prop:adMult}.\end{proof}

	\begin{remark}\label{rmk:coend} It is possible to describe $\underline{\End}_{\cA^{\bt2}}(\mathbf{1}_{\cA})$ explicitly, via the co-end construction~\cite{Lyubashenko1995,Lyubashenko1994,Majid1995}; we have:
\begin{equation}\underline{\End}(_\cA\mathbf{1}_\cA) = \left(\bigoplus_{V\in\comp\mathcal{A}} V^*\bt V\right) \Big/ \langle \operatorname{Im}(\id_{W^*} \bt \phi - \phi^*\bt\id_V) \,\,|\,\, \phi:V\to W \rangle.\label{coend}\end{equation}
If $\mathcal{A}$ is semi-simple, we may make this more explicit by choosing a representative $X$ of every simple isomorphism class of object and write simply,
$$\underline\End_{\cA^{\bt 2}}(\mathbf{1}_\cA) \cong \bigoplus_{X} X^*\bt X.$$
When expressed through the canonical maps $\iota_U:U^*\bt U\to\underline{\End}(_\cA\mathbf{1}_\cA)$, the algebra structure is the tautological one, composed with the braiding:
{\footnotesize$$V^*\bt V \ot W^*\bt W = (V^*\ot W^*) \bt (V \ot W) \xrightarrow{\sigma_{V^*,W^*}} (W^*\ot V^*) \bt (V \ot W) \xrightarrow{\iota_{V\ot W}} \underline{\End}(_\cA\mathbf{1}_\cA).$$}

\end{remark}

\begin{definition}\label{def:REA} Let $\oa:=T(\underline{\End}_{\cA^{\bt 2}}(\mathbf{1}_{\cA}))$ denote the algebra in $\cA$ obtained by applying the tensor product $T:\cA\bt\cA\to\cA$.
\end{definition}

The algebra $\oa$ will play an important role throughout the rest of the paper, as the basic building block for quantum algebras computing factorization homology.  As a special case of monadicity for base change, we have:

\begin{theorem}[Monadicity for traces]\label{thm:tr-monad}
We have an equivalence of categories,
$$\cA\underset{\cA\bt \cA}{\bt} \cA\simeq \oa\modu_\cA.$$
\end{theorem}
\begin{proof}

We have shown in Proposition \ref{prop:1AA} that $\mathbf{1}_\cA$ is a progenerator for the $\cA\bt\cA$-action.  Note that we have an evident isomorphism of functors $\act_{\mathbf{1}_\cA}\cong T$.  Hence applying monadicity for module categories, Theorem \ref{monadicitymodules}, we have an equivalence of $\cA\bt\cA$-module categories,
$$\cA_{\cA\bt\cA}\simeq\underline{\End}(_\cA\mathbf{1}_\cA)\modu_{\cA\bt\cA}.$$
We may now apply monadicity for base change, Corollary \ref{base-change}, to the tensor functor $(T,\sigma):\cA\bt\cA\to \cA$ obtain the asserted equivalence of categories.

We note that the tensor product functor plays two distinct roles in this proof: it gives the $\cA\bt\cA$-module structure on $\cA$ allowing us to apply monadicity for module categories and it serves, together with its braiding, as a tensor functor $\cA\bt\cA\to\cA$ to which we apply monadicity for base change.

\end{proof}

\begin{remark}In particular, let us note in passing that Theorem \ref{thm:tr-monad} gives an explicit computation of the zeroeth Hochschild homology, or monoidal trace, of a braided tensor category.  This computation of the Hochschild homology appears to be new, though its antecedents can be seen in Lyubasheko and Majid's earlier works \cite{Majid1995,Majid1993b,Lyubashenko1995} which pre-dated the notion of Hochschild homology of tensor categories.\end{remark}

Clearly, the description of $\underline{\End}_{\cA^{\bt 2}}(\mathbf{1}_{\cA})$ implies a similar presentation,

\begin{equation}\oa = \left(\bigoplus_{V\in\comp(\mathcal{A})} V^*\ot V\right) \Big/ \langle \operatorname{Im}(\id_{W^*} \ot \phi - \phi^*\ot\id_V) \,\,|\,\, \phi:V\to W \rangle,\label{coend-br}\end{equation}

and in the case $\cA$ is semi-simple, we may write:

$$\oa \cong \bigoplus_{X} X^*\ot X.$$

The multiplication map $\mu:\oa\ot\oa\to \oa$ may be expressed entirely internally to $\cA$.  Using the canonical maps $\iota_V:V^*\otimes V\to \oa$, we may write the multiplication as follows:
$$(V^*\ot V) \ot (W^* \ot W) \xrightarrow{\sigma_{V^*\ot V,W^*}} (W^*\ot V^*)\ot (V\ot W) \xrightarrow{\iota_{V\ot W}} \oa.$$

See Section~\ref{sec:examples} for a thorough discussion of the algebra $\oa$, and its many appearances in low-dimensional topology and representation theory.

For later use, we will also need the following modification of Theorem~\ref{thm:tr-monad}.  For an integer $k\in\ZZ$, let $I^{(k)}:\cA\rightarrow\cA$ be the tensor functor whose underlying functor is the identity, and whose tensor structure is given for $X,Y \in \cA$ by $(\sigma_{Y,X}\sigma_{X,Y})^k$. We observe that $I^{(0)}$ is nothing but the identity functor, and that a balancing on $\cA$ is by definition an isomorphism of tensor functors $\id=I^{(0)}\cong I^{(1)}$, which clearly implies that $I^{(k)}\cong \id$ as well.

Let $\cA^{(k)}$ be the right $\cA \bt \cA$ bimodule whose underlying category is $\cA$, and whose module structure is induced by
\begin{equation}\label{eq:twistedMul}
	\cA \bt \cA \xrightarrow{\id \bt I^{(k)}} \cA \bt \cA \xrightarrow{T} \cA.
\end{equation}

Let $\oak{k}$ denote the algebra in $\cA\bt\cA$ which is the image of $\un_\cA$ under the right adjoint of~\eqref{eq:twistedMul}. Crucially, as objects $\oak{k}\cong T^R(1)\cong\oa$ but the multiplication is twisted by the double braidings. We also note that if $\cA$ is balanced there is a canonical algebra isomorphism $\oak{k}\cong \oa$. Then, Theorems \ref{thm:tr-monad} and \ref{base-change} combine to give:

\begin{corollary}[Monadicity for twisted traces]\label{cor:twisted}
We have an equivalence of categories,
$$\cA^{(k)}\underset{\cA\bt \cA}{\bt} \cA\simeq \oak{k}\modu_\cA.$$
\end{corollary}

\section{Computing factorization homology of punctured surfaces}\label{sec:surfaces}
This section contains the first main result of this paper, that is an explicit computation of the factorization homology of a braided tensor category on any punctured surface.  In particular, we treat both the cases of framed surfaces, with coefficients in braided tensor categories, and of oriented surfaces, with coefficients in balanced braided tensor categories, separately.
\subsection{The distinguished object}\label{sec:dist}
Recall from Section~\ref{sec:pointed} that the inclusion 
\[
\emptyset \longrightarrow M
\]
of the empty manifold into any surface induces a canonical functor
\[
	\Vect_{\K}\longrightarrow \int_M\cA.
\]
\begin{definition}
The image of the tensor unit $\K\in \Vect$ under the above functor will be called the \emph{distinguished object} or \emph{quantum structure sheaf}, and denoted by $\cO_{\cA,M}\in \int_M\cA$.
\end{definition}

We note that by construction, the distinguished object is mapped to the distinguished object by the functor on factorization homology induced by an embedding of manifolds.  In this section we explicitly compute the functor associated to certain particular embeddings:

\begin{proposition}\label{pi1}
Let $X=(\RR^2)^{\sqcup k}$ be a disjoint union of $k$ 2-disks.  We have:
\begin{enumerate}
\item We have an isomorphism $\cO_{\cA,X} \cong \un_{\cA^{\boxtimes k}}$.
\item The functor on factorization homology induced by any embedding $X \rightarrow \RR^2$ is isomorphic to the tensor functor $\cA^{\boxtimes k} \rightarrow \cA$.
\item Any two embeddings of $X$ into a path connected manifold $M$ give rise to isomorphic functors, and any such factors through the tensor functor $\cA^{\boxtimes k}\rightarrow \cA$.
\end{enumerate}
\end{proposition}
\begin{proof}
Parts (1) and (2) are contained in the equivalence between locally constant factorization algebras on $\RR^2$ and $E_2$-algebras. Part (3) follows from the fact that $\operatorname{Emb}(X,M)$ is path connected, and the fact that any such embedding can be factored through an embedding of $X$ into a bigger disk. 
\end{proof}

\subsection{Moduli algebras}\label{sec:moduliAlgebra}

Let $S$ be a punctured surface, together with a choice of an interval along the boundary. We define the following algebra:
\begin{definition}
The moduli algebra of $S$ is
$A_S:=\underline{\End}_\cA(\cO_{\cA,S}),$
where the internal endomorphisms are defined with respect to the $\cA$-action on $\int_S\cA$ given by the chosen interval.
\end{definition}

The main goal of this section is to describe a combinatorial and explicit presentation of $A_S$.  Punctured surfaces may be indexed by combinatorial data called \emph{gluing patterns} $P$.  In this section, we will give explicit presentations of algebras $a_p$ in $\cA$, whose categories $a_p\modu_\cA$ of modules in $\cA$ describe the factorization homology of the associated marked, punctured surface $\Sigma(P)$.

To focus on the main ideas, \emph{we will first assume that $\cA$ is balanced}.  This simply allows us to choose the most convenient framing on a given surface for which to do computations.  In Section~\ref{sec:framed} we remove this assumption, and state the parallel results for framed surfaces, which follow easily by modifying proofs in the oriented case.

\begin{definition}
A \emph{gluing pattern} is a bijection,
$$P:\{1,1',\ldots, n,n'\}\xrightarrow{\sim} \{1,\ldots 2n\},$$
such that $P(i)<P(i')$, for all $i = 1,\ldots, n$.
\end{definition}
A convenient notation to specify a gluing pattern $P$ is by enumerating the tuple $(P(1),P(1'),\ldots, P(n),P(n'))$.  We highlight for future use the following gluing patterns,
\begin{equation}Ann:=(1,2), \qquad T^2\backslash D^2 := (1,3,2,4).\label{eqn:building-blocks}\end{equation}
By convention, we allow the null gluing pattern $D^2:\emptyset\to\emptyset$.  See the Figures~\ref{fig:annulus-pants} and~\ref{torus-gluing-pattern}, in Section \ref{sec:examples} for examples of gluing patterns.

\begin{remark}
Recall that every (connected) oriented surface with boundary can be described as the thickening of a (connected) so-called fat graph or ribbon graph, i.e. a graph equipped with a cyclic ordering at each vertex. A \emph{ciliated graph} is defined similarly, but with a choice of a linear instead of cyclic ordering at each vertex. Ciliated graphs enter the definition of the Fock-Rosly Poisson structure on the representation variety of the underlying surface~\cite{Fock1999}, and of its quantization by Alekseev~\cite{Alekseev1993}. A gluing pattern is the same as a ciliated graph with one vertex. We use this notion for the sake of simplicity, but we note that both Definition~\ref{def:glued-alg} and our main result Theorem~\ref{thm:main} has a straightforward analog for an arbitrary ciliated graph.
\end{remark}

\begin{figure}[h]\centering
\includegraphics[width=5in]{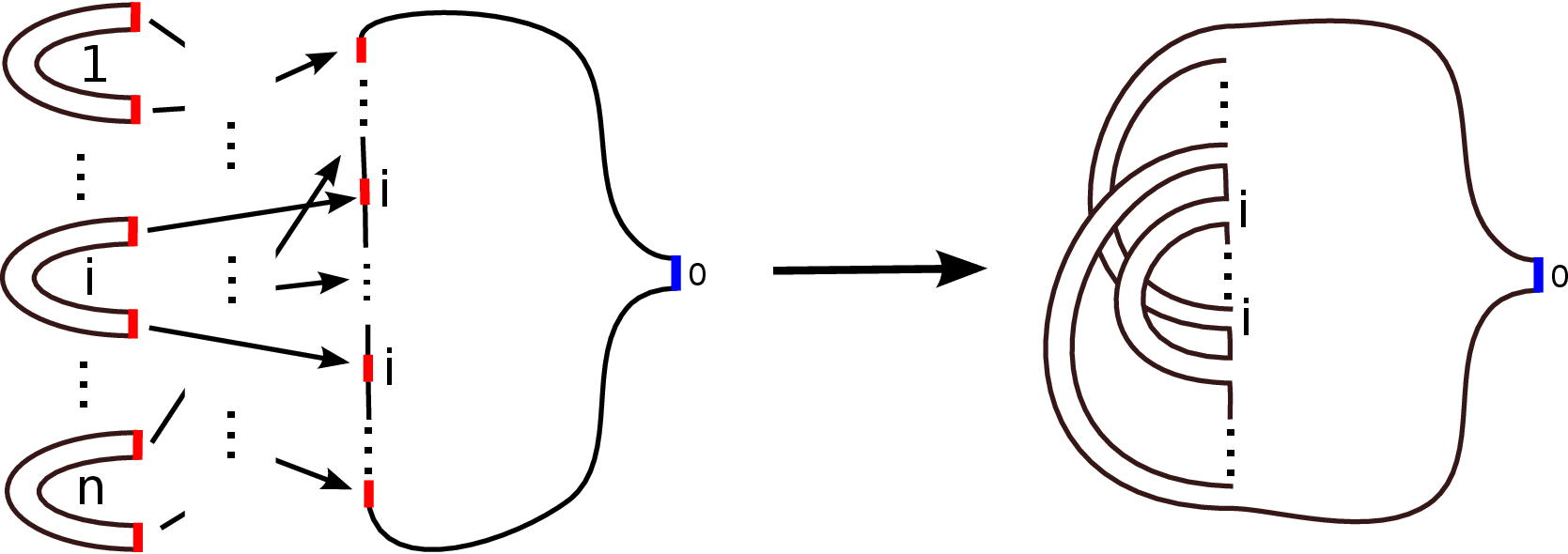}
\caption{The punctured surface $\Sigma(P)$ constructed from a gluing pattern $P$.}
\label{gluing-pattern}
\end{figure}

A gluing pattern $P$ determines a surface $\Sigma(P)$, with a single marked boundary interval $0$, as follows:  we begin with a disk $D^2$ with $2n+1$ boundary intervals labeled $0,1,\ldots 2n$, and then glue each of $n$ handles $H_i$, with marked intervals, $i$ and $i'$ to $P(i)$ and $P(i')$, respectively.

\begin{definition} The \emph{disjoint union} of gluing patterns $P:\{1,1',\ldots,n,n'\}\to \{1,\ldots 2n\}$ and $Q:\{1,1',\ldots,m,m'\}\to \{1,\ldots,2m\}$ is
$$P\sqcup Q:\{1,1',\ldots, m+n, (m+n)'\} \to \{1,\ldots, 2(m+n)\},$$
$$i,i' \mapsto \left\{ \begin{array}{ll}P(i), P(i'), & \textrm{if $i\leq n$}\\ Q(i-n)+2n, Q((i-n)')+2n, & \textrm{if $i>n$} \end{array}\right.$$
\end{definition}

See Figure \ref{fig:geng} for an example, yielding an arbitrary genus $g$ surface with $r$ punctures from disjoint unions of the basic building blocks \eqref{eqn:building-blocks}.

\begin{definition}
We say that handles $H_i$ and $H_j$, with $i<j$ are:
\begin{itemize}
\item \emph{positively linked} if $P(i) < P(j) < P(i') < P(j')$,\\
\emph{negatively linked} if $P(j) < P(i) < P(j') < P(i')$,
\item \emph{positively nested} if $P(i) < P(j) < P(j') < P(i')$,\\
\emph{negatively nested} if $P(j) < P(i) < P(i') < P(j')$,
\item \emph{positively unlinked} if $P(i) < P(i') < P(j) < P(j')$,\\
\emph{negatively unlinked} if $P(j) < P(j') < P(i) < P(i')$,
\end{itemize}
where the sign is $(+)$ if $i<j$, and $(-)$ if $i>j$.
\end{definition}

\begin{definition} We define the crossing morphisms,
$$L, N, U: \oa\ot\oa\to \oa\ot\oa,$$
as follows (diagrams are read from bottom to top):
$$\begin{tikzpicture}
\draw (1.5,-5) node{Linked crossing $L$};
\braid[
 line width=2pt,
 style strands={1,2}{blue},
] (braid2) at (0,0) s_2^{-1} s_1-s_3 s_2 ;
\draw (braid2-4-e) node[below left]{$\oa$};
\draw (braid2-2-e) node[below left]{$\oa$};
\draw (braid2-4-s) node[above left]{$\oa$};
\draw (braid2-2-s) node[above left]{$\oa$};

\draw (6.5,-5) node{Nested crossing $N$};
\braid[
 line width=2pt,
 style strands={1,2}{blue},
] (braid2) at (5,0) s_2^{-1} s_1^{-1}-s_3 s_2 ;
\draw (braid2-4-e) node[below left]{$\oa$};
\draw (braid2-2-e) node[below left]{$\oa$};
\draw (braid2-4-s) node[above left]{$\oa$};
\draw (braid2-2-s) node[above left]{$\oa$};

\draw (11.5,-5) node{Unlinked crossing $U$};
\braid[
 line width=2pt,
 style strands={1,2}{blue},
] (braid3) at (10,0) s_2 s_1-s_3 s_2 ;
\draw (braid3-4-e) node[below left]{$\oa$};
\draw (braid3-2-e) node[below left]{$\oa$};
\draw (braid3-4-s) node[above left]{$\oa$};
\draw (braid3-2-s) node[above left]{$\oa$};
\end{tikzpicture}$$
\end{definition}
\begin{remark}
	Those diagrams make sense because $\oa$, as an object in $\cA$, can be written as a colimit over objects of the form $V^* \ot V$ (see Remark~\ref{rmk:coend} and Equation~\eqref{coend-br}).
\end{remark}
\begin{remark}
	The unlinked crossing operator $U$ is nothing but the braiding of $\cA$ applied to $\oa\ot \oa$.
\end{remark}
\begin{definition}\label{def:glued-alg} For a gluing pattern of rank $n$, we define the algebra $a_{P}$ to be the object $\oa^{\ot n}$ in $\cA$. The multiplication is defined as follows: denote by $\oa^{(i)}$ the $i$th copy of $\oa$ inside $\oa^{\ot n}$. Then for each pair of indices $1\leq i<j \leq n$ the restriction of the multiplication to $\oa^{(i)}\ot \oa^{(j)}\subset a_P$ is given by:
$$\oa^{(i)}\ot\oa^{(j)} \ot \oa^{(i)} \ot \oa^{(j)} \xrightarrow{\id\ot C \ot \id} \oa^{(i)}\ot\oa^{(i)}\ot\oa^{(j)}\ot\oa^{(j)} \xrightarrow{m\ot m}\oa^{(i)}\ot\oa^{(j)},$$ 
where $C= L^{\pm 1}$ (resp. $N^{\pm 1}, U^{\pm 1})$, if $H_i$ and $H_j$ are ($\pm$)-linked (resp. ($\pm$)-nested, ($\pm$)-unlinked), and $m$ denotes the multiplication on $\oa$.
\end{definition}

\begin{remark}We note in passing that the notations $A_\Sigma$ and $a_P$ for the algebras defined in t	his section are potentially ambiguous, as they depend essentially on the choice of braided tensor category $\cA$.  However, we will suppress this dependence from our notation, as it will always be clear from context.\end{remark}

	\begin{proposition}\label{prop:disjoint}
	Let $P=P_1\amalg P_2$ be the disjoint union of two gluing pattern. The algebra $a_P$ is the braided tensor product of $a_{P_1}$ and $a_{P_2}$.
\end{proposition}
\begin{proof}
By construction, any pair of a handle in $P_1$ and a handle in $P_2$ is unlinked, hence the cross relations between the corresponding $\oa$ factors are those of the braided tensor product. The result thus follows from the hexagon axioms.
\end{proof}

The following is the main result of this section:
\begin{theorem}\label{thm:main}
	Let $\cA$ be an abelian rigid balanced braided tensor category in $\LFP$. We have an isomorphism of algebras $A_{\Sigma(P)}\cong a_P$, and an equivalence of categories,
$$\int_{\Sigma(P)}\cA \simeq a_{P}\modu_{\cA}\simeq A_{\Sigma(P)}\modu_\cA$$
\end{theorem}

\begin{proof}
We may deform the attaching disk of $\Sigma(P)$ in such a way that the $0$-marked interval is on the right, and all other marked intervals are on the left. 

We have $\int_D\cA\simeq\cA$ as a category; however the markings on $D$ induce a $\cA^{\bt 2n}$-$\cA$-bimodule structure,
$$(a_1\bt \cdots \bt a_{2n}) \bt b \bt c \mapsto a_1\ot \cdots \ot a_{2n} \ot b \ot c,$$
which bimodule we denote by ${}_{2n}\cA_\cA$. 

For each handle $H_i$, we have $\int_{H_i}\cA \simeq \cA_{\cA\cA}$, the category $\cA$ with its right regular $\cA\boxtimes \cA$-module structure, $b\bt (a_1\bt a_2) \mapsto b\ot a_1 \ot a_2$.

Likewise, given $n$ handles $H_1, \ldots H_n$ we have $\int_{H_1\cup \cdots \cup H_n}\cA \simeq \cA^{\bt n}$ as a category.  We make this a right $\cA^{\bt 2n}$-module using $P$:
\[
	(a_1 \bt \cdots \bt a_{n'}) \bt (b_1\bt\ldots \bt b_{2n}) \mapsto (a_1\ot b_{P(1)}\ot b_{P(1')})\bt \cdots \bt  (a_{n}\ot b_{P(n)} \ot b_{P(n')}).
\]
We denote this right $\cA^{\bt 2n}$-module by $\cA^P$.

We note that the module structures on $A^P$ and $_{2n}\cA_\cA$ are precisely those induced by the markings on the left hand side of Figure \ref{gluing-pattern}.  Thus, by the excision property for factorization homology, we have:

\begin{align}\label{eq:rel1}
	\int_{\Sigma(P)}\cA \simeq \cA^{P} \underset{\cA^{\boxtimes 2n}}{\bt} \,_{2n}\cA_\cA.
\end{align}
Let $\tau_P \in S_{2n}$ be the permutation obtained by precomposing $P$ with the map
\[
	\{1,\dots,2n\}\rightarrow \{1,1',\dots, n,n'\}
\]
defined by
\[
i \mapsto
\begin{cases}
	i/2& \text{if }$i$\text{ is even}\\
	(i/2)'& \text{if }$i$\text{ is odd}
\end{cases}
\]
Applying Proposition~\ref{prop:1AA}, we may identify $\cA^P$ with the category of modules in $\cA^{\bt 2n}$ for an algebra $\underline{\End}(\mathbf{1}_{\cA\cA})^P$ obtained by applying $\tau_P$ to $\underline{\End}(\mathbf{1}_{\cA\cA})^{\bt n}$. 

The bimodule $_{2n}\cA_\cA$ is simply that induced by the iterated tensor product functor $T^{2n}: \cA^{\bt 2n}\to \cA$, which itself carries the structure of a tensor functor using formula \eqref{eq:tensor}.  Hence we may apply Theorem \ref{base-change} to conclude,
\begin{align}\label{eq:rel2}
	\int_{\Sigma(P)}\cA \simeq T^{2n}(\underline{\End}(\mathbf{1}_{\cA\cA})^P)\modu_\cA,
\end{align}
as a right $\cA$-module category.

Note that equations \eqref{eq:rel1} and \eqref{eq:rel2}, combined with Theorem \ref{base-change}, yield an isomorphism,
$$\widetilde{a_P}\cong A_S.$$
Let us now explain how to identify $\tilde{a}_P:=T^{2n}(\underline{\End}(\mathbf{1}_{\cA\cA})^P)$ with the algebra $a_P$. First, we consider the subalgebras $\oa^{(i,i')}=\underline{\End}(\mathbf{1}_{\cA_{P(i)}\cA_{P(i')}})$, and their images $\oa^{(i)}:=T(\oa^{(i,i')})$.  The multiplication induces an isomorphism,

$$m: \oa^{(1)} \ot \cdots \ot \oa^{(n)} \to \tilde{a}_p,$$
on the level of objects.  It remains only to compute the pairwise cross relations between factors.  Note that $\oa^{(i,i')}$ and $\oa^{(j,j')}$ commute in $\cA^{\bt 4}$, because they occupy different tensor factors.  Hence we have the following commutative diagram:
\begin{equation}\label{batman}\resizebox{.9\hsize}{!}{\xymatrixcolsep{1pc}
\xymatrix{
& T^4(\oa^{(i,i')}\ot\oa^{(j,j')})\ar[ddr]^{T^4(m)} &\underset{can}{=} & T^4(\oa^{(j,j')}\ot \oa^{(i,i')}) \ar[ddl]_{T^4(m)}&\\
\oa^{(i)}\ot\oa^{(j)}\ar[drr]_m \ar[ur]^{J_{ij}} & && & \oa^{(j)}\ot\oa^{(i)}\ar[ul]_{J_{ji}}\ar[dll]^m\\
&& a_P &&
}}\end{equation}

The ordering on the tensor factors in $\cA^P$, and hence the value of $J_{ij}$, depends on the gluing pattern.  For example, in the positively linked case, as indicated in Figure \ref{linked}, we have $J_{12}=\id\ot\sigma\ot\id,$ and  $J_{21}=(\sigma\ot\sigma)\circ(\id\ot\sigma\ot\id),$ hence:
$$m|_{\oa^{(2)}\ot \oa^{(1)} }  = m|_{\oa^{(1)}\ot\oa^{(2)}}\circ J_{12}^{-1}J_{21} = m|_{\oa^{(1)}\ot\oa^{(2)}}\circ L^+_{12},$$
as claimed.  The other five cases follow a similar computation.

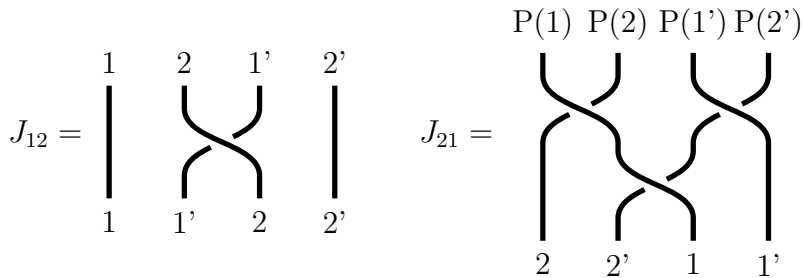
\begin{figure}[h]
$$
J_{12} =
\raisebox{-0.5\height}{\begin{tikzpicture}
\braid[
 number of strands=4,
 line width=2pt,
] (braid2) at (0,0) s_2 ;
\draw (braid2-4-e) node[below]{2'};
\draw (braid2-2-e) node[below]{2};
\draw (braid2-3-e) node[below]{1'};
\draw (braid2-1-e) node[below]{1};
\draw (braid2-4-s) node[above]{2'};
\draw (braid2-2-s) node[above]{2};
\draw (braid2-3-s) node[above]{1'};
\draw (braid2-1-s) node[above]{1};
\end{tikzpicture}}\qquad
J_{21} =
\raisebox{-0.5\height}{\begin{tikzpicture}
\braid[
 line width=2pt,
] (braid2) at (0,0) s_1-s_3 s_2 ;
\draw (braid2-4-e) node[below]{2'};
\draw (braid2-2-e) node[below]{2};
\draw (braid2-3-e) node[below]{1'};
\draw (braid2-1-e) node[below]{1};
\draw (braid2-4-s) node[above]{P(2')};
\draw (braid2-2-s) node[above]{P(2)};
\draw (braid2-3-s) node[above]{P(1')};
\draw (braid2-1-s) node[above]{P(1)};
\end{tikzpicture}}$$
\caption{In the linked case, we have $P(1,1',2,2') = (1,3,2,4)$.}\label{linked}
\end{figure}

\end{proof}

\subsection{The framed case}\label{sec:framed}

Let us now suppose that $\cA$ is a braided tensor category in $\LFP$, which is not necessarily balanced, and that accordingly $S$ is a \emph{framed} surface, equipped with a marked boundary interval.  In this section, we will compute the algebra $A_S$.  To this end, we need to refine the notion of gluing pattern.  In place of the ``handle and comb" decomposition of the previous section, we have instead a ``coil and comb" decomposition.  Here a ``coil" is a framing of each handle, relative to its two marked boundary intervals.  Clearly $S$ may be presented as glued from a collection of coils onto a common disk.  We call a framing of a disk relative to its two marked boundary intervals a coil, since we can induce the framing from the blackboard framing, as for handles, but only under an immersion, which resembles a coil; see Figure~\ref{fig:geng}.  We encapsulate the underlying combinatorics as follows:

\begin{definition}
A coiling of rank $n$ is a function $\xi:\{1,\ldots n\}\to \ZZ$.  A coiled gluing pattern is a pair consisting of gluing pattern $P$, and a coiling $\xi$, each of rank $n$.
\end{definition}

Accordingly, let $\Sigma(P,\xi)$ denote the surface obtained by gluing in the coils framed by each $\xi(k)$ along the pattern $P$ (see Figure~\ref{fig:geng}).

\begin{definition} For a coiled gluing pattern $(P,\xi)$ of rank $n$, we define the algebra $a_{P,\xi}$ to be the object $\oak{\xi(1)}\ot \cdots \ot \oak{\xi(n)}$ in $\cA$. The cross relations between different factors is the same as in Definition \ref{def:glued-alg}.
\end{definition}

\begin{theorem}
Let $\cA$ be an abelian rigid braided tensor category in $\LFP$. We have an isomorphism of algebras $A_{\Sigma(P,\xi)}\cong a_{P,\xi}$, and an equivalence of categories,
$$\int_{\Sigma(P,\xi)}\cA \simeq a_{P,\xi}\modu_{\cA}\simeq A_{\Sigma(P,\xi)}\modu_\cA$$
\end{theorem}
\begin{proof}
We simply note that the $\cA\bt\cA$-module associated to each coil is a copy of $\cA^{(k)}$, so that in place of Theorem \ref{thm:tr-monad}, we instead appeal to Corollary~\ref{cor:twisted}.  Crucially, the computation of the cross relations in Theorem \ref{thm:main} remain entirely unchanged.
\end{proof}

\begin{remark}
We are now in a position to compare concretely two related facts:  on the one hand, if $\cA$ is in fact balanced, then the algebra $a_{P,\xi}$ is independent of a coiling, because the theory $\cA$ defines is oriented.  On the other hand, a balancing is precisely an $\cA\bt\cA$-module equivalence between $\cA$ and $\cA^{(1)}$ (and hence between any two twists $\cA^{(k)}$).  The choice of a ribbon element therefore determines a canonical isomorphism between $\oa$ and $\oak{k}$ for any $k$, and hence between $a_P$ and $a_{P,\xi}$.
\end{remark}

\subsection{Mapping class group actions}\label{sec:mcg}

For the sake of simplicity we return again to oriented surfaces (and hence we assume that $\cA$ is balanced); everything can be extended in a straightforward way to the framed setting. Let $M=\Sigma_{g,n+1}$ be a surface of genus $g$ with $n+1$ circle boundary components. By functoriality of factorization homology, every orientation-preserving self-diffeomorphism $f$ of $M$ induces an auto-equivalence of $\int_M \cA$ as a pointed category, i.e. an equivalence of categories $f_*$ from $\int_M \cA$ to itself, together with an isomorphism $f_*(\cO_{\cA,M})\rightarrow \cO_{\cA,M}$. Likewise, any isotopy between two such diffeomorphisms induces a natural isomorphism between the corresponding functors intertwining the isomorphisms from $\cO_{\cA,M}$ with its images. In other words, there is an action of the truncation $\pi_{\leq 1}(Diff(M))$ on the category $\int_M \cA$ for which the distinguished object has a canonical structure of an homotopy fixed point.

Clearly, one obtains the same group if one considers diffeomorphisms and isotopies preserving a small annular neighborhood $\tilde b$ of $b$. Fix once and for all an embedding of a small disk inside $\tilde b$. To this embedding corresponds a functor $\iota:\cA \rightarrow \int_M \cA$ which is isomorphic to the functor $\act_{\cO_{\cA,M}}$. Fix such an isomorphism once and for all. Given a diffeomorphism preserving $\tilde b$, the corresponding functor $F$  commute with $\iota$ strictly by definition. In particular we have that $F(\cO_{\cA,M})$ is \emph{equal}, rather than isomorphic, to $\cO_{\cA,M}$. Therefore, the pointed structure on $F$ induces an automorphism of the object $\cO_{\cA,M}$. Now an isotopy between two such diffeomorphisms, itself preserving $\tilde b$, induces an isomorphism between the corresponding pointed functors. In particular the associated automorphisms of $\cO_{\cA,M}$ are equal, which implies the following:
\begin{proposition}
	Fix an interval on the boundary of $M$ and let $A_M$ be the corresponding algebra in $\cA$. Then there is a canonical action of $\Gamma_{n,1}^g$ on $A_M$ by algebra automorphisms.
\end{proposition}

\subsection{Braid group representations}\label{sec:braid}
Let $M$ be a (connected) framed (resp, oriented) surface and fix a framed (resp. oriented) embedding $\iota:(\RR^2)^{\sqcup n} \hookrightarrow M$. It follows from the formalism of factorization homology that we have an action of the fundamental group $\pi_1(\operatorname{Emb}((\RR^2)^{\sqcup n},M), \iota)$ on the associated functor
	\[
		F_{\iota}: \cA^{\boxtimes n} \rightarrow \int_M \cA,
	\]
which is clearly $S_n$-equivariant. This fundamental group is nothing but the pure braid group of $M$, in the framed case, or the framed braid group in the oriented case. 

	For $n=1$ and for any choice of $\iota$, the functor $F_{\iota}$ is isomorphic to the free module functor $A_M \ot-$ and for $n \geq 1$ one can arrange so that $F_{\iota}$ coincides with
	\[
V_1 \bt \dots \bt V_n \rightarrow A_M \ot V_1 \dots \ot V_n.
	\]

	Therefore for each $\gamma$ in the (framed) braid group $B_n(M)$ of $M$ we obtain natural isomorphisms,
	\[
		A_M \ot V_1 \dots \ot V_n \longrightarrow A_M \ot V_{\bar\gamma(1)} \dots \ot V_{\bar\gamma(n)}
	\]
where $\bar\gamma$ is the image of $\gamma$ through $B_n(M)\rightarrow S_n$.  In particular if all of the $V_i$ are equal, we have a representation of the surface braid group on the tensor product.  In the next section, we will compare these with several constructions of braid group actions on quantum algebras already appearing in the literature.

\section{Examples}\label{sec:examples}
Let us now catalog some well-known algebras which arise as instances of the algebras $A_S$ constructed in the preceding section.  First, we need to recall some preliminaries about the best-understood source of examples of braided tensor categories, namely the categories $\Rep H$ of locally finite-dimensional modules for a quasi-triangular Hopf algebra $H$, which we define below.

Recall that for a Hopf algebra $H$, the category $H\modu$ of left $H$-modules has a natural tensor category structure via the coproduct. Similarly, there is a notion of a quasitriangular
Hopf algebra (See e.g.~\cite{Kassel1995} for an exposition), which encodes the extra structure to
determine a braided tensor category structure on $H\modu$.  The additional data consists of a choice of invertible element $R\in H\ot H$, called the universal $R$-matrix, with a list of properties yielding that $\sigma_{V,W} = \tau_{V,W}\circ R: V\ot W\to W\ot V$ defines a braiding on $H\modu$, where $\tau_{V,W}$ swaps the tensor components.   
Finally, Hopf algebras, as opposed to bialgebras, have an antipode which implies that there is a natural $H$-module structure on the vector space dual of any $H$-module. There are two, however, closely related issues if we consider the category of all $H$-modules: the first is that often, in examples, the $R$-matrix lives in a completion of $H \ot H$ so that its action is usually not defined on the whole of $H\modu$. Secondly, $H\modu$ is not rigid in the sense of Definition~\ref{def:rigid}. Both issues are fixed by considering instead the category $\Rep H$ of locally finite modules, i.e. the ind-completion of the category of finite dimensional modules. Equivalently, this is the category of comodules over the Sweedler dual $H^{\circ}$ of $H$. Accordingly, we will call quasi-triangular any Hopf algebra which has the property that $\Rep H$ is braided, or equivalently such that $H^{\circ}$ is coquasi-triangular.

\paragraph{\textbf{Convention}} We will fix throughout this section a quasi-triangular Hopf algebra, and let $\cA=\Rep H$ denote its braided tensor category of locally finite-dimensional modules, as above.
\begin{example}
The quantum group $U_q(\mf{g})$ is a one-parameter deformation of the universal enveloping algebra $U(\mf{g})$ of a reductive algebraic group (see \cite{Kassel1995} for a survey).  The universal $R$-matrix is an infinite sum in the so-called Serre generators, so a priori only the category $\Repq G = U_q(\mf{g})\modu_{lf}$ of locally-finite dimensional modules has a well-defined braiding.
\end{example}

In the remainder of this section, we will highlight a number of isomorphisms between special cases of the algebras $A_S$ and well-known algebras in the representation theory of quantum groups.

\subsection{Braided dual of a quasi-triangular Hopf algebra}
Given a Hopf algebra $H$, its restricted dual $H^\circ$ is the subalgebra of the full linear dual $H^*$ spanned by so-called \emph{matrix coefficients} of finite dimensional representations.  The matrix coefficient $c_{f,v}$ associated to a finite-dimensional representation $V$ of $H$, a vector $v\in V$ and a covector $f\in V^*$ is the linear functional defined by $c_{f,v}(h) = f(hv)$.

\begin{example}
The restricted dual Hopf algebra to a quantum group $U_q(\mf{g})$ is called the Fadeev-Reshetikhin-Takhtajan (FRT) algebra.  It is a one-parameter deformation of the algebra $\cO(G)$ of functions on the simply connected group $G$ integrating $\mf{g}$, along the so-called Sklyanin Poisson bracket.
\end{example}

It was observed early on by Majid that the restricted dual of an Hopf algebra $H$, in particular the FRT algebra, was not a module-algebra under the coadjoint action
\[
	Ad(h)(f):=x\mapsto f(h^{(1)}xS(h^{(2)}))
\]
where $h \in H, f\in H^*$ and $\Delta(h)=\sum h^{(1)}\ot h^{(2)}$. In particular quantizations of adjoint-equivariant structures were not possible in that framework.  Majid therefore introduced an algebra called the \emph{braided dual} algebra $\widetilde H$ to a quasi-triangular Hopf algebra.  This defines a new algebra structure on the restricted dual (the same underlying vector space), with the key difference being that the multiplication of matrix coefficients twisted by a natural expression in the $R$-matrices. 

By a result of Lyubashenko~\cite{Lyubashenko1995} (see also~\cite{Majid1995}), the braided dual $\widetilde{H}$ is isomorphic to the algebra $\oa$ constructed via the CoEnd construction~\eqref{coend-br}.

\begin{example} The braided dual of $U_q(\mf{g})$ is a quantization of the algebra $\cO(G)$, along the so-called Semenov-Tian-Shansky Poisson bracket (so it is also sometimes denoted by $\cO_q(G)$). This Poisson structure is known to coincide with the Fock-Rosly Poisson structure on $G\cong\RV(Ann)$ (see~\cite{Fock1999} and Section~\ref{sec:CharVar} ) and the braided dual coincides with Alekseev moduli algebra of the annulus~\cite{Alekseev1993}.

In this case, the braided dual goes by many names: it has appeared in the literature as the ``equivariantized quantum coordinate algebra" of Majid~\cite{Majid1995}, ``the reflection equation algebra" (see Remark \ref{rmk:basis} below)~\cite{Donin2003,Donin2002}, and the ``quantum loop algebra" of Alekseev and Schomerus~\cite{Alekseev1993,Alekseev1996,Alekseev1996a}. A comprehensive reference is the text \cite{Klimyk1997}. 
\end{example}

Another important feature of quasi-triangular Hopf algebras is that there is a natural linear map
\begin{align*}
	H^\circ &\to H,\\
	f&\mapsto (f \ot \id)(R_{21}R_{12}) 
\end{align*}
which is in fact an algebra morphism for the algebra structure of $\widetilde H$ $(=H^\circ)$. A quasi-triangular Hopf algebra is called \emph{factorizable} if this map is injective, and it turns out that quantum groups are factorizable in that sense~\cite{Reshetikhin1988} (this should be thought of as a quantum analog of the non-degeneracy of the Killing form). The image of the above map coincides with $U'_q(\mf g)$ (see e.g.~\cite{Baumann1998}), where $U'_q(\mf g)$ denotes the ad-locally finite part of $U_q(\mf g)$, i.e. the vectors which generate a finite-dimensional orbit under the quantum adjoint action.

\subsection{The factorization homology of the annulus}
The annulus is homeomorphic to $\Sigma(Ann)$, where we recall that $Ann(1)=1, Ann(1')=2$. 
\begin{figure}[h]\centering
\includegraphics[width=5in]{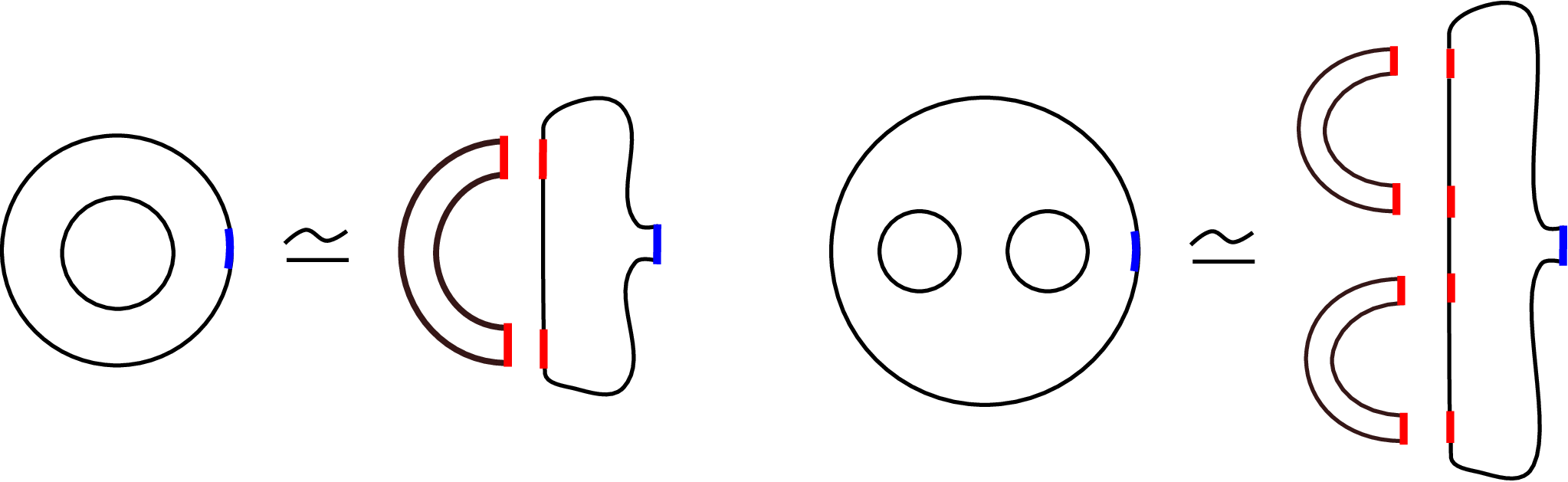}
\caption{The annulus constructed from a gluing pattern yields $A_{Ann}\cong a_{Ann}\cong \oaH$.  The pair of pants yields the braided tensor product $a_{Pan}\cong \oaH \widetilde{\ot}\oaH$.}\label{fig:annulus-pants}
\end{figure}

\begin{corollary}\label{refl-eqn-theorem}
Let $\cA=\Rep H$.  We have the following identifications:
\[
	A_{Ann} \cong \oaH, \qquad \int_{Ann} \Rep H \simeq \oaH\modu_{\Rep H}.
\]\end{corollary}

The braided dual appearing here was used in~\cite{Donin2003a,Donin2003,Donin2002} in order to produce universal solutions of the so-called reflection equation: it is observed in \emph{loc. cit.} that the canonical element $L\in \widetilde H \ot H$ (recall that $\widetilde H=H^*$ as vector spaces) satisfies the following equation
\begin{equation}\label{eq:DKM}
	(\id \ot \Delta)(L)=(R^{1,2})^{-1}L^{0,2}R^{1,2}L^{0,1}
\end{equation}
in $\widetilde H \ot H^{\ot 2}$ and where $\widetilde H$ has index 0, and is in fact the universal solution of this equation in the sense that if $A$ is any algebra and $L_A \in A \ot H$ any solution of~\eqref{eq:DKM}, then there is a unique algebra morphism $\widetilde H \rightarrow A$ mapping $L$ to $L_A$. Equation~\eqref{eq:DKM} implies that $L$ satisfies the \emph{reflection equation}:
\begin{equation}\label{eqn:REA}
	R_{21}L_{0,2}R_{12}L_{0,1}=	L_{0,1}R_{21}L_{0,2}R_{1,2}.
\end{equation}

\begin{remark}\label{rmk:basis}Equation~\eqref{eqn:REA} is called the reflection equation, and it expresses the commutation relations in the braided dual.  It is for this reason that one may call the braided dual Hopf algebra (and its generalizations to arbitrary braided tensor categories) the reflection equation algebra.

Choosing bases for objects $V,W$, we have elements $(a_V)^i_j:=v^i\ot v_j$ and $(a_W)^k_l =w^k\otimes w_l$ of $\oaH$, and the collection of all such elements spans $\oaH$.  The commutation relations \eqref{eqn:REA} between them may be written explicitly as follows:

$$\sum_{k,l,m,p}R^{ij}_{kl}(a_V)^l_mR^{mk}_{op}(a_W)^p_r = \sum_{s,t,u,v}(a_W)^i_sR^{sj}_{tu}(a_V)^u_vR^{vt}_{or}.$$
 \end{remark}

Roughly speaking, the reflection equation is related to the braid group of the annulus the same way the quantum Yang-Baxter equation is related to the braid group of the plane~\cite{Kulish1993}. Crucially, the axioms of braided tensor categories, and in particular those of quasi-triangular Hopf algebras, not only provide representations of the braid group of the plane, but imply that those are naturally compatible with the operadic structure of the braid group, i.e. with the operation of ``doubling a strand''. Equation~\eqref{eq:DKM} plays a similar role for the braid group of the annulus.

The annular braid group representations of Donin, Kulish, and Mudrov can now be recovered as follows.  As explained in Section~\ref{sec:braid}, factorization homology of any framed manifold carries representations of the corresponding braid groups, naturally compatible with strand doubling operations. One notes that the carrying space for these representations is the tensor product $\widetilde{H}\ot V_1\ot\ldots \ot V_n$, where $V_1,\ldots, V_n$ may be taken arbitrarily from the $\Rep H$.  This is precisely the carrying space of the Donin-Kulish-Mudrov representations.  It follows from their respective universal properties (or alternatively, by a direct computation), that the two representations of the annular braid group on this space coincide.  We can state this more formally, following the notation from Section \ref{sec:braid}
\begin{proposition}\label{prop:AnnBrRec}
We have a canonical isomorphism of annular braid group representations, 
$$\operatorname{For}(F_\iota(V_1\bt \cdots\bt V_n)) \cong \widetilde{H}\ot V_1\ot\ldots \ot V_n,$$
where $\operatorname{For}$ denotes the forgetful functor from $\int_{Ann}\Rep H$ to $\Vect$, equipped with the action by isotopies of disk inclusions, and where the action on the RHS is the Donin-Kulish-Mudrov action. 
\end{proposition}
Further implications of this correspondence are taken up in the sequel, \cite{Ben-Zvi2016}.

\begin{remark}Let us stress here that for the strand doubling operations to be well-defined one needs the underlying manifold to be framed (the framing determines in which direction to push the copy of the strand when doubling). Since there are $\ZZ$ many framings on an annulus, there are in fact \emph{a priori} $\ZZ$ many versions of equation~\eqref{eq:DKM} and as many versions of the braided dual.  This is discussed in detail in Section \ref{sec:twisted-framings} below.
\end{remark}

\subsection{Factorization homology of several-punctured disk}

Let us consider now the $r$-punctured disc, which we denote $\Sigma_{0,r+1}$ (following the usual convention to view the disc itself as a once-punctured sphere).  We may regard this simply as a disjoint union of the annular gluing pattern, and hence Proposition~\ref{prop:disjoint} implies:
 
\begin{corollary}
Let $\cA=\Rep H$.  We have the following identifications:
\[
	A_{\Sigma_{0,r+1}} \cong \oaH^{\ot r}, \qquad \int_{\Sigma_{0,r+1}} \Rep H \simeq (\oaH)^{\ot r} \modu_{\Rep H},
\]
where $\oaH^{\ot r}$ denotes the braided tensor power of the algebra $\oaH$.
\end{corollary}

In~\cite{Donin2003a}, Donin--Kulish--Mudrov observed that the $n$th braided tensor power of the reflection equation algebra provides representations of the braid group of an $n$-punctured disk.  As in Proposition \ref{prop:AnnBrRec}, we note that the carrying spaces of both representations coincide, and that each satisfies the same universal property with respect to strand doubling, so that the braid group representations they induce are isomorphic.

\subsection{Quantum differential operators and the elliptic double}

Another basic algebra associated to a quantum group $U_q(\mf g)$ is the ring $\cD_q(G)$ of quantum differential operators on the group $G$.  Recall that the classical algebra of differential operators on $G$ can be constructed as a semi-direct product $\cD(G)=U(\mf g)\ltimes \cO(G)$ where $\mf g$ acts on $\cO(G)$ via right invariant differential operators. This construction can be generalized to any Hopf algebra by considering the so-called Heisenberg double $H\ltimes H^*$, where $H$ acts by the left coregular action; in the case $H=U_q(\mf g)$ one obtains this way a deformation of $\cD(G)$.

It turns out to be slightly more natural to consider instead the semi-direct product $\cD_q(G)=U'_q(\mf g)\ltimes U_q(\mf g)^*$ where $U'_q(\mf g)$ is the ad-locally finite part of $U_q(\mf g)$. This is a deformation of $\cO(G\times G)$ in the direction of a certain Poisson structure also introduced by Semenov-Tian-Shansky~\cite{Semenov1994} as a Poisson-Lie version of the cotangent bundle $T^*G$, and coincides with the Fock-Rosly Poisson structure on $\RV(T^2\backslash D)$ (see again~\cite{Fock1999} and Section~\ref{sec:CharVar}).

It is proven in~\cite{Varagnolo2010} that one obtains the same algebra using the braided instead of the ordinary dual, i.e. there is an equivariant algebra isomorphism $\cD_q(G)\cong U'_q(\mf g)\ltimes \cO_q(G)$ (this holds for any quasi-triangular Hopf algebra). 

	In \cite{Brochier2017} an algebra called the elliptic double $\cD_H$ was introduced, which carries representations of the braid group of the punctured torus. It was shown there (following \cite{Varagnolo2010}) that if $H$ is factorizable, then the factorizing isomorphism extends to an isomorphism between the elliptic and the Heisenberg double (though in general they can be distinct).  We note that the elliptic double for $U_q(\mf g)$ also coincides with the moduli algebra of the punctured torus of Alekseev \cite{Alekseev1993}

\subsection{Factorization homology of the punctured torus}\label{sec:elliptic}

In fact the elliptic double can be regarded as an instance of Theorem \ref{thm:main}, as follows.  The punctured torus $T^2\backslash D$ is homeomorphic to $\Sigma(P)$, where $P(1)=1, P(1')=3, P(2)=2, P(2')=4$.  See Figure \ref{torus-gluing-pattern}.

\begin{figure}[h]\centering
\includegraphics[width=5in]{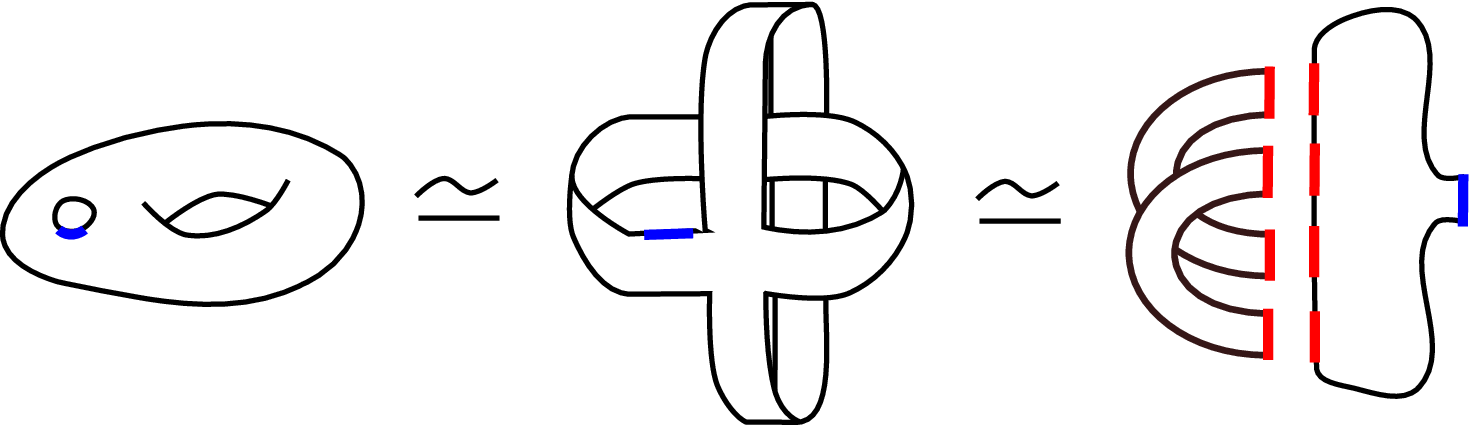}
\caption{The punctured torus constructed from a gluing pattern yields $A_{T^2\backslash D}\cong \cD_H$.}
\label{torus-gluing-pattern}
\end{figure}

Hence $A_{T^2\backslash D^2} \cong a_{T^2\backslash D^2} \cong\cD_H$, where $\cD_H$ is the algebra in $\cA=\Rep H$ which as an object is $\oaH\ot \oaH$, and whose cross relations, given in Theorem~\ref{thm:main}, coincide with those of the double  (see below).  In summary, we have:

\begin{corollary}
Let $\cA=\Rep H$.  We have the following identifications:
\[
	A_{T^2\backslash D^2} \cong \mathcal{D}_H, \qquad \int_{S} \Rep H \simeq \mathcal{D}_{H}\modu_{\Rep H}
\].
\end{corollary}

\begin{remark} Following Remark~\ref{rmk:basis} we can write the cross relations of the elliptic double more explicitly as follows.  First we choose bases for representations $V$ and $W$ of $U_q(\mf g)$.  We let $L_V, L_W$ denote the $N\times N$ (resp. $M\times M$) matrix $(L_V)^k_l = v^k\bt v_l$, (resp. $(L_W)^m_n = w^m\bt w_n$), with the entries regarded as elements of $\oaH^{13}$.  Let $D_V$ and $D_W$ denote the same matrices, but with elements regarded instead in $\oaH^{24}$.  We can write the commutation relations in matrix form:
$$L_VR_{V,W}D_W = R_{V,W}D_WR_{W,V}L_VR_{V,W},$$
or even more explicitly,
$$\sum_{j,m} l^i_jR^{jk}_{lm}\partial^m_n = \sum_{o,p,r,t,u,v} R^{ik}_{op}\partial^p_rR^{ro}_{tu}l^u_vR^{vt}_{ln},$$
where we have omitted the labels $V,W$ on $R$-matrices for ease of notation.
\end{remark}

In~\cite{Brochier2017}, the constructions of Donin-Kulish-Mudrov were generalized to the case of a punctured torus, and certain universal representations of the punctured torus braid group were constructed, compatible in an analogous way with the strand-doubling operation.

Using this universal property, it is proved in \emph{loc. cit.} that there is an action by algebra automorphisms of the mapping class group $\Gamma_{0,1}^1$, which is the universal central extension $\widetilde{SL_2(\ZZ)}$ of the modular group, on $\cD_q(G)$. Recall that $\widetilde{SL_2(\ZZ)}$ is generated by $X,Y$ with relations
	\begin{align*}
		X^4&=(XY)^3 & (X^2,Y)&=1.
	\end{align*}
	Using the presentation given in Section~\ref{sec:elliptic}, the action of the standard generators $X,Y$ is given by:
\begin{align*}
	X\cdot L_V&=D_V & X\cdot D_V&=D_VL_V^{-1}D_V^{-1}\\
	Y\cdot L_V&=L_V & Y\cdot D_V & = D_V L_V^{-1}.
\end{align*}
We call the automorphism induced by the action of $X$ the ``quantum Fourier transform'' because it degenerates to the classical Fourier transform on $D(\mf g)$.

Once again, following Proposition \ref{prop:AnnBrRec}, the representations of the braid group constructed in~\cite{Jordan2009,Brochier2017} and in Section~\ref{sec:braid} can be shown to coincide.  It follows from Section~\ref{sec:mcg} that $\cD_q(G)$ carries an action of the mapping class group of the once-bordered torus, and this coincides with the mapping class group action in~\cite{Brochier2017}.
\subsection{Factorization homology of the $r$-punctured surface of genus $g$}

The $g$-fold disjoint union of the once-punctured torus gluing pattern yields a gluing pattern for the once-punctured surface of genus $g$.  The disjoint union of this with the $r-1$-fold disjoint union of the annular gluing pattern yields a gluing pattern for $\Sigma_{g,r}$, the $r$-punctured genus $g$ surface.
Accordingly, we have
\begin{corollary}\label{prop:main-sigmagr}
Let $\cA=\Rep H$.  We have the following identifications:
\[
	A_{\Sigma_{g,r}} \cong \cD_H^{\ot g}\ot \oaH^{\ot r-1}, \qquad \int_{\Sigma_{g,r}} \Rep H \simeq \cD_H^{\ot g}\ot \oaH^{\ot r-1}\modu_{\Rep H},
\]
where $\ot$ denotes the braided tensor product of algebras in $\cA$.
\end{corollary}
\begin{figure}[h]\center
\includegraphics[height=2in]{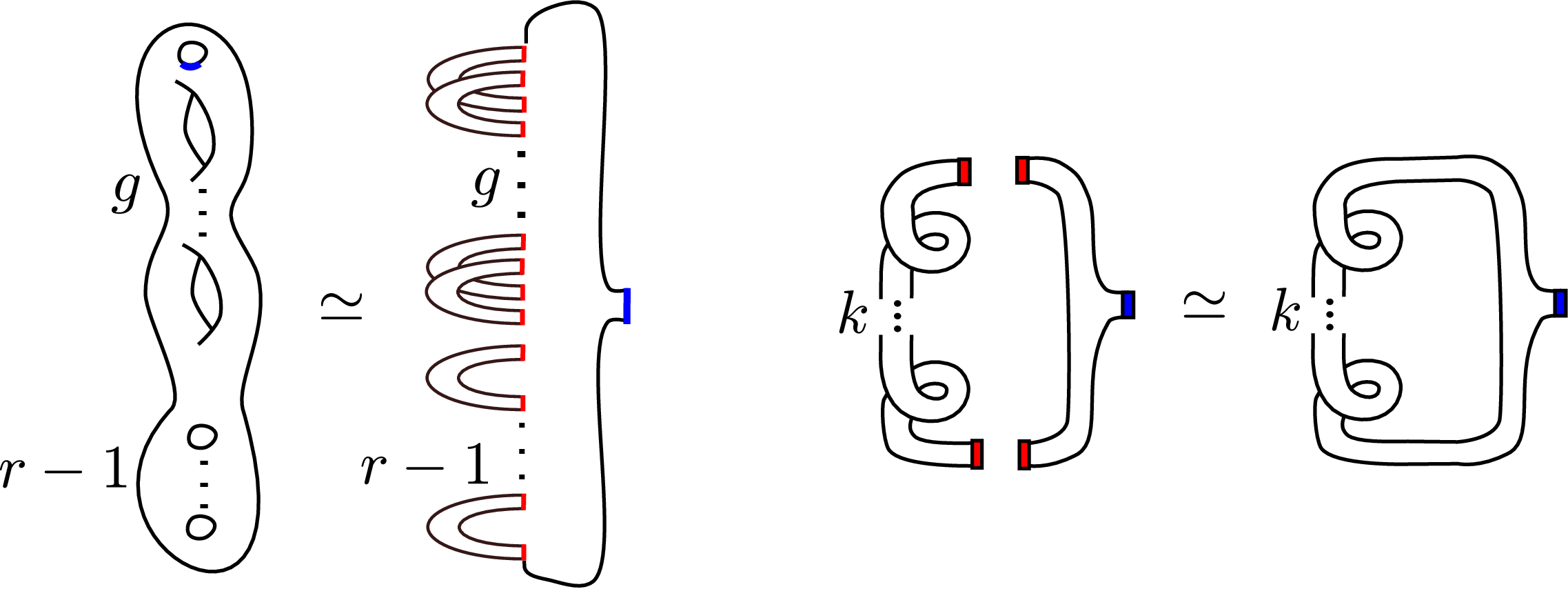}
\caption{At left, the $r$-punctured, genus $g$ surface from a gluing pattern yields the iterated tensor product of $g$ copies of $\cD_H$ and $r-1$ copies of $\oaH$.  At right, the annulus with framing of index $k$ is obtained by a coiled gluing pattern with a single coiled handle.}\label{fig:geng}
\end{figure}
\subsection{Factorization homology of non-blackboard-framed annuli}\label{sec:twisted-framings}
When we treat the case of $k$-framed annuli, we replace the algebra $\oaH\cong\oa$ with its twisted version $\oaHk{k}=\oak{k}$.

\begin{definition}[\cite{Brochier2017}]\label{def:kth-braided-dual}
For an integer $k$, the $k$-twisted braided dual Hopf algebra $\widetilde H_{k}$ is $H^{\circ}$ as a vector space, with multiplication given by
\[
x \cdot y = m ( R_{13}R_{14}(R_{31}R_{13})^k\rhd (x \ot y))
\]
where $m$ above denotes the multiplication of $H^{\circ}$.
\end{definition}
It follows immediately from their definitions that the case we have isomorphisms $\oak{k}\cong \oaHk{k}$.

In~\cite{Brochier2017}, it was explained that for any quasi-triangular Hopf algebra in fact the twisted braided duals $\widetilde H_k$ and $\widetilde H_{k+2}$ are isomorphic via the so-called Drinfeld element, and that when $H$ is moreover a ribbon Hopf algebra (i.e. the category $\Rep H$ is a ribbon tensor category), then the $\widetilde H_k$ are isomorphic for all $k$.

An annulus with framing indexed by $k\in\mathbb{Z}$ is framed-diffeomorphic to surface $
\Sigma(P,\xi)$, where $P$ is the trivial gluing pattern on one handle, and $\xi$ denotes the $k$th coiling of the handle.  See Figure \ref{fig:geng}.  Let $Ann_k$ denote the annulus equipped with the framing indexed by integer $k$.  We have the following straightforward modification of Theorem \ref{thm:main}:

\begin{corollary}\label{prop:main-twisted}
Let $\cA=\Rep H$.  We have the following identifications:

\[
	A_{S} \cong \tilde{H}_k, \qquad \int_{Ann_k} \Rep H \simeq \tilde{H}_k\modu_{\Rep H}.
\]
\end{corollary}
\begin{proof}
All that is required to modify the proof of the main theorem is to note that the role played by the right-regular module $\cA_{\cA^{\bt 2}}$ is now played instead by the same module, where the action of one of the $\cA$ factors is twisted by the $k$th power of the double-braiding. 
\end{proof}

\section{Character varieties and quantization}\label{sec:CharVar}

In this section we explain how our formalism recovers the classical character stack and the character variety, and provides a quantization of the Atiyah--Bott Poisson structure on the latter.
\subsection{Classical character varieties from $\Rep G$}
Recall that $\uch_G(S)$ and $\ch_G(S)$ are respectively the character stack and the character variety defined in the introduction. If $S$ is a connected surface of genus $g$ with $n>0$ points removed, then $\pi_1(S)$ is the free group on $2g+n-1$ generators, hence
\[
	\uch_G(S)\cong G^{2g+n-1}/G,
\]
where the (stacky) quotient is taken with respect to the diagonal adjoint action. In that case, the category of quasi-coherent sheaves on $\uch_G(S)$ is easily identified with the category of $\cO(G^{2g+n-1})$-modules in $\Rep G$. Comparing with Theorem~\ref{thm:main}, we have:
\begin{theorem}
If $S$ is a punctured surface, then we have an equivalence of categories,
$$\QCoh(\uch_G(S))\simeq \int_S\Rep G,$$
between the category of quasi-coherent sheaves on $\uch_G(S)$, and the factorization homology of $S$ with coefficients in $\Rep G$.
\end{theorem}

There is a natural identification $\uch_G(S)\cong Map(S,BG)$ where $BG$ is the classifying stack of $G$. We denote by $\Rep_\h G$ the formal version of $\Repq  G$, defined as in Section~\ref{qg-defs} but using the formal quantum group $U_\h(\mf g)$ instead. Hence, if $X$ is a topological space of dimension $0,1,2$, factorization homology with coefficients in $\Rep_\h G $ should be thought of as a quantization of the functor
\[
	QCoh(Map(-,BG)).
\]

In~\cite{Pantev2013} (see also~\cite{Toen2014}) the authors prove that the classifying stack $BG$ has a canonical 2-shifted symplectic structure induced by the Killing form on $\mf g$. It gives rise by integration/pull-back to a 0-shifted Poisson structure on $\uch_G(S)$ (symplectic when $S$ is closed) which is a stacky version of the Poisson structure on the affine algebraic variety $\ch_G(S)$. The category $\int_{S}\cA$ is a quantization of this structure. 

The main goal of this section is to show directly that our constructions also provide a quantization of the Poisson algebra of functions on the categorical quotient $\ch_G(S)$, in the sense of deformation quantization.

\subsection{Gluing patterns, fat graphs and Poisson structures}
Let $S$ be a punctured surface of genus $g$, and choose a gluing pattern $P$ for $S$ with $n$ handles, and let $\cA=\Rep_\h G$.  Recall that we have associated to this data an algebra $a_P$ in $\Rep_\h G$. The category $\Rep_\h G$ has a tautological strict tensor functor to the category of vector spaces, hence the algebra $a_P$ can be identified with an associative algebra, under this functor.  A gluing pattern can be viewed as an instance of a ciliated fat graph in the sense of~\cite{Fock1999}, with only one vertex. Recall that $G$ has a standard Poisson Lie group structure. To a ciliated graph, Fock-Rosly attach a Poisson structure on $G^{2g+n-1}$ and shows that the adjoint action of $G$ in this space is Poisson-Lie. Their main result is the following:
\begin{theorem}
	The inherited Poisson structure on $G^{2g+n-1}/G$ does not depend on the choice of the underlying ciliated fat graph, and coincides with the Atiyah--Bott Poisson structure.
\end{theorem}
The main result of this section is then:

\begin{theorem}\label{thm:charVar}
	The algebra $a_P$ is a quantization of the Fock--Rosly Poisson structure on $G^{2g+n-1}$. Its subalgebra of invariants, $\Hom_{U_{\h}(\mf g)}(\CC,a_p)$, does not depend on the choice of $P$ and is a quantization of the Atiyah--Bott Poisson structure on $S$.
\end{theorem}
Before the proof, let us reformulate Fock--Rosly's construction in a way convenient for our purpose (we refer the reader to the original paper~\cite{Fock1999} or the survey~\cite{Audin1997} for details). Fix a ciliated graph $\Gamma$ with one vertex, let $\Sigma$ be the corresponding surface and label the edges from 1 to $n$ in the order determined by $\Gamma$. Since $\Gamma$ has only one vertex, the Fock--Rosly Poisson bracket is defined on
\[
\cO(G)^{\ot n}=\cO(\Hom(\pi_1(\Sigma),G)).
\]
Moreover, it is enough to compute the Poisson bracket $\{f,g\}$ for $f \in \cO(G)^{(i)}$ and $g \in \cO(G)^{(j)}$ for $i\leq j$, which in turn is determined from the Poisson brackets coming from graphs with one vertex and one or two edges.

For $x \in \mf g$, we will denote by $x^l,x^r$ the action of $x$ on $G$ by left-invariant and right invariant vector field respectively, and we set $x^{ad}=x^r-x^l$. Let $r \in \mf g^{\ot 2}$ be the classical limit of the quantum R-matrix and set $\rho=\frac12 (r^{1,2}-r^{2,1})$ and $t=\frac12(r^{1,2}+r^{2,1})$. Let $\pi_{STS}$ be the bivector field
\[
\pi_{STS}=\rho^{ad,ad}+t^{r,l}-t^{l,r}.
\]
It induces on $G$ a Poisson structure which has been introduced by Semenov--Tian--Shansky~\cite{Semenov1994}. We will denote by $G_{STS}$ the variety $G$ equipped with this Poisson structure. Then $G_{STS}$ is a Poisson Lie variety under the adjoint action of the Poisson-Lie group $G$.

We can then extract the following from~\cite[Propositions 3 and 5]{Fock1999}:
\begin{theorem}\label{thm:FR}
Let $\Gamma$ be a ciliated graph with one vertex and $n$ edges. The corresponding Poisson structure on $\cO(G)^{\ot n}$ is induced by the bivector field
\[
\sum_i \pi_{STS}^{(i)}+\sum_{i<j} (\pi_{ij}-\pi_{ji})
\]
where it is understood that $\pi_{ij}$ is a 2-tensor acting on the $i$th component of the first factor and the $j$th component of the second factor of $G^{2n}=G^n\times G^n$, defined by:
\[
\pi_{ij}=
\begin{cases}
\pm (r^{ad,ad}) & \text{if $i,j$ are $\pm$ unlinked}\\
\pm (r^{ad,ad}+2t^{r,l}) & \text{if $i,j$ are $\pm$ linked}\\
\pm (r^{ad,ad}-2t^{r,r}+2t^{r,l}) & \text{if $i,j$ are $\pm$ nested }
\end{cases}
\]
\end{theorem}
\subsection{Quantization via $\Rep_\h G$}
\begin{proof}[Proof of Theorem~\ref{thm:charVar}]
The following is well known (see e.g.~\cite{Donin2003,Mudrov2006}):
\begin{theorem}\label{thm:STS}
	The reflection equation algebra $\cO(\Rep_\h G)$ is a flat deformation of $\cO(G)$, quantizing $G_{STS}$.
\end{theorem}
Moreover, the left $U_{\h}(\mf g)^{\ot 2}$ action on the reflection equation algebra becomes, at the classical limit, the left $U(\mf g)^{\ot 2}$ action on $\cO(G)$, where the action of the first copy of $U(\mf g)$ is induced by
\[
\mf g\ni x \mapsto x^r
\]
and the second action is induced by
\[
\mf g \ni x \mapsto -x^l
\]
(since left invariant vector fields act on the right, the minus sign turns this into a left action). Hence to prove the claim it is enough to check that say for $i<j$ the bivector field $\pi_{ij}$ coincides with the quasi-classical limit of the action of $L_{ij}$, $N_{ij}$ or $U_{ij}$ if the handles $i,j$ are linked, nested or unlinked respectively. We prove this in the nested case, the other cases are similar: dropping indices for the sake of clarity, we have
\begin{equation}
N = \tau_{12,34}\circ R_{1,34} (R^{-1})_{34,2} = \tau_{12,34}\circ R_{12,34}(R_{34,2}R_{2,34})^{-1},
\end{equation}
and hence,
\begin{equation}\label{eq:qcl}
\frac{\tau_{12,34}\circ N-1}{\h} \mod (\h)=r^{12,34} - 2 t^{1,3}-2t^{1,4}.
\end{equation}
where we used that $r^{i,j}+r^{j,i}=2t^{i,j}$ and
\begin{align*}
r^{ij,k}&=r^{i,k}+r^{j,k} & r^{i,jk}=r^{i,j}+r^{i,k}.
\end{align*}
The right hand side of~\eqref{eq:qcl} acts on $G\times G$ via the bivector field
\[
r^{ad,ad}-2t^{r,r}+2t^{r,l}
\]
as required.
\end{proof}


\end{document}